\documentclass[11pt]{article}
\usepackage{amsfonts}
\usepackage{amssymb,amsmath,amsthm}
\usepackage{latexsym}
\usepackage{graphics}
\usepackage{epic}
\usepackage{epsfig}
\usepackage{psfrag}
\usepackage{dsfont}
\usepackage{enumitem}
\usepackage{yhmath}
\usepackage{pst-pdf}
\usepackage{stmaryrd}

\numberwithin{equation}{section}
\def\PP{\mathbb{P}}
\def\QQ{\mathbb{Q}}

\def\EE{\mathbb{E}}
\def\11{\mathds{1}}

\def\E{\mathbb{E}}
\def\P{\mathbb{P}}
\def\R{\mathbb{R}}
\def\Q{\mathbb{Q}}

\def\d{\partial}
\def\Z{\mathbb{Z}}

\newtheorem{thm}{Theorem}[section]
\newtheorem{lem}[thm]{Lemma}
\newtheorem{cor}[thm]{Corollary}

\newtheorem{prop}[thm]{Proposition}

\theoremstyle{remark}
\newtheorem{rem}{Remark}
\newtheorem{exa}{Example}

    \def\restriction#1#2{\mathchoice
                  {\setbox1\hbox{${\displaystyle #1}_{\scriptstyle #2}$}
                  \restrictionaux{#1}{#2}}
                  {\setbox1\hbox{${\textstyle #1}_{\scriptstyle #2}$}
                  \restrictionaux{#1}{#2}}
                  {\setbox1\hbox{${\scriptstyle #1}_{\scriptscriptstyle #2}$}
                  \restrictionaux{#1}{#2}}
                  {\setbox1\hbox{${\scriptscriptstyle #1}_{\scriptscriptstyle #2}$}
                  \restrictionaux{#1}{#2}}}
    \def\restrictionaux#1#2{{#1\,\smash{\vrule height .8\ht1 depth .85\dp1}}_{\,#2}}

\begin{document}

\title{Exponential convergence to quasi-stationary distribution and $Q$-process}

\author{Nicolas Champagnat$^{1,2}$, Denis Villemonais$^{1,2}$}

\footnotetext[1]{Universit\'e de Lorraine, IECN, Campus Scientifique, B.P. 70239,
  Vand{\oe}uvre-l\`es-Nancy Cedex, F-54506, France} \footnotetext[2]{Inria, TOSCA team,
  Villers-l\`es-Nancy, F-54600, France.\\
  E-mail: Nicolas.Champagnat@inria.fr, Denis.Villemonais@univ-lorraine.fr}

\maketitle

\begin{abstract}
  For general, almost surely absorbed Markov processes, we obtain necessary and sufficient conditions for exponential convergence to
  a unique quasi-stationary distribution in the total variation norm. These conditions also ensure the existence and exponential
  ergodicity of the $Q$-process (the process conditioned to never be absorbed). We apply these results to one-dimensional birth and
  death processes with catastrophes, multi-dimensional birth and death processes, infinite-dimensional population models with
  Brownian mutations and neutron transport dynamics absorbed at the boundary of a bounded domain.
\end{abstract}



\section{Introduction}
\label{sec:intro}

Let $(\Omega,({\cal F}_t)_{t\geq 0},(X_t)_{t\geq 0},(P_t)_{t\geq 0},(\PP_x)_{x\in E\cup \{\partial\}})$ be a
time homogeneous Markov process with state space $E\cup\{\partial\}$~\cite[Definition III.1.1]{rogers-williams-00a}, where
$(E,{\cal E})$ is a measurable space and $\partial\not\in E$. We recall that $\PP_x(X_0=x)=1$, $P_t$ is the transition function of the process satisfying the usual measurability assumptions and Chapman-Kolmogorov equation. The family $(P_t)_{t\geq 0}$ defines a semi-group of operators on the set ${\cal B}(E\cup\{\d\})$ of bounded Borel functions on $E\cup\d$ endowed with the uniform norm.
  We will also denote by $p(x;t,dy)$ its transition kernel, \textit{i.e.} $P_t f(x)=\int_{E\cup\{\d\}} f(y) p(x;t,dy)$ for all $f\in{\cal B}(E\cup\{\d\})$.
For all probability measure $\mu$ on $E\cup\{\partial\}$, we will use the notation
$$
\PP_\mu(\cdot):=\int_{E\cup\{\partial\}}\PP_x(\cdot)\mu(dx).
$$
We shall denote by $\EE_x$ (resp.\ $\EE_\mu$) the expectation corresponding to $\PP_x$ (resp.\
$\PP_\mu$).

We consider a Markov processes absorbed at $\partial$. More precisely, we assume that $X_s=\partial$ implies $X_t=\partial$ for all
$t\geq s$. This implies that
$$
\tau_\partial:=\inf\{t\geq 0,X_t=\partial\}
$$
is a stopping time. We also assume that $\tau_\partial<\infty$ $\PP_x$-a.s.\ for all $x\in E$ and for all $t\geq 0$ and $\forall x\in E$, $\PP_x(t<\tau_\partial)>0$.

Our first goal is to prove that Assumption (A) below is a necessary and sufficient criterion for the existence of a unique quasi-limiting distribution $\alpha$ on $E$ for the process 
$(X_t,t\geq 0)$, i.e. a probability measure $\alpha$ such that for all probability measure $\mu$ on $E$ and all $A\in{\cal E}$,
\begin{align}
\label{eq:QLD}
\lim_{t\rightarrow+\infty}\PP_\mu(X_t\in A\mid t<\tau_\partial)=\alpha(A),
\end{align}
where, in addition, the convergence is exponential and uniform with respect to $\mu$ and $A$.
In particular, $\alpha$ is also the unique \emph{quasi-stationary distribution}~\cite{MV12},
i.e. the unique probability measure $\alpha$ such that $\PP_\alpha(X_t\in\cdot\mid
t<\tau_\partial)=\alpha(\cdot)$ for all $t\geq 0$.

\paragraph{Assumption~(A)}
There exists a probability measure $\nu$ on $E$ such that
\begin{itemize}
\item[(A1)] there exists $t_0,c_1>0$ such that for all $x\in E$,
  $$
  \PP_x(X_{t_0}\in\cdot\mid t_0<\tau_\partial)\geq c_1\nu(\cdot);
  $$
\item[(A2)] there exists $c_2>0$ such that for all $x\in E$ and $t\geq 0$,
  $$
  \PP_\nu(t<\tau_\partial)\geq c_2\PP_x(t<\tau_\partial).
  $$
\end{itemize}

\begin{thm}
  \label{thm:QSD}
  Assumption~(A) implies the existence of a probability measure $\alpha$ on $E$ such that, for any initial  distribution $\mu$,
  \begin{align}
  \left\|\PP_\mu(X_t\in\cdot\mid t<\tau_\partial)-\alpha(\cdot)\right\|_{TV}\leq
  2(1-c_1c_2)^{\lfloor t/t_0\rfloor},
  \end{align}
  where $\lfloor\cdot\rfloor$ is the integer part function and $\|\cdot\|_{TV}$ is the total variation norm.
  
  Conversely, if there is uniform exponential convergence for the total variation norm in~\eqref{eq:QLD}, then Assumption (A) holds true.
\end{thm}

Stronger versions of this theorem and of the other results presented in the introduction will be given in the next sections.


The quasi-stationary distribution describes the distribution of the process on the event of non-absorption. It is well known
(see~\cite{MV12}) that when $\alpha$ is a quasi-stationary distribution, there exists $\lambda_0>0$ such that, for all $t\geq 0$,
\begin{align}
\label{eq:alphaPtbis}
\PP_{\alpha}(t<\tau_{\partial})=e^{-\lambda_0 t}.
\end{align}
The following proposition characterizes the limiting behaviour of the absorption probability for other initial distributions.
\begin{prop}
\label{prop:2intro}
There exists a non-negative function $\eta$ on $E\cup\{\d\}$, positive on $E$ and vanishing on $\d$, such that
\begin{align*}
\mu(\eta)=\lim_{t\rightarrow\infty} e^{\lambda_0 t}\P_\mu(t<\tau_\d),
\end{align*}
where the convergence is uniform on the set of probability measures $\mu$ on $E$.
\end{prop}

Our second goal is to study consequences of Assumption~(A) on the behavior of
the process $X$ conditioned to never be absorbed, usually referred to as the $Q$-process (see~\cite{Athreya1972} in discrete time and for example~\cite{Cattiaux2008} in continuous time).

\begin{thm}
\label{thm:Q-processintro}
Assumption~(A) implies that the family $(\QQ_x)_{x\in E}$ of
probability measures on $\Omega$ defined by
$$
\QQ_x(A)=\lim_{t\rightarrow+\infty}\PP_x(A\mid t<\tau_\partial),\ \forall A\in{\cal F}_s,\ \forall s\geq 0,
$$
is well defined and
the process $(\Omega,({\cal F}_t)_{t\geq 0},(X_t)_{t\geq
  0},(\QQ_x)_{x\in E})$ is an $E$-valued homogeneous Markov process.
In addition, this process admits the unique invariant distribution
\begin{align*}
\beta(dx)=\frac{\eta(x)\alpha(dx)}{\int_E \eta(y)\alpha(dy)}
\end{align*}
and, for any $x\in E$,
\begin{align*}
\left\|\Q_{x}(X_t\in\cdot)-\beta\right\|_{TV}\leq 2(1-c_1c_2)^{\lfloor t/t_0\rfloor}.
\end{align*}
\end{thm}

The study of quasi-stationary distributions goes back to~\cite{Yaglom1947} for branching processes
and~\cite{Darroch1965,Seneta1966,Darroch1967} for Markov chains in finite or denumerable state spaces, satisfying irreducibility
assumptions.
 In these works, the existence and the
convergence to a quasi-stationary distribution are proved using spectral properties of the generator of the absorbed Markov process.
This is also the case for most further works. For example, a extensively developed tool to study birth and death processes is based
on orthogonal polynomials techniques of~\cite{Karlin1957}, applied to quasi-stationary distributions
in~\cite{Good1968,Cavender1978,vanDoorn1991}. For diffusion processes, we can refer to~\cite{Pinsky1985} and more
recently~\cite{CCLMMS09,Cattiaux2008,Littin2012}, all based on the spectral decomposition of the generator. Most of these works only
study one-dimensional processes, whose reversibility helps for the spectral decomposition. Processes in higher dimensions were
studied either assuming self-adjoint generator in~\cite{Cattiaux2008}, or using abstract criteria from spectral theory like
in~\cite{Pinsky1985,Collet2011a} (the second one in infinite dimension). Other formulations in terms of abstract spectral theoretical
criteria were also studied in~\cite{KnoblochPartzsch2010}. The reader can refer to~\cite{MV12,collet-martinez-al-13b,vanDoorn2013}
for introductory presentations of the topic.

Most of the previously cited works do not provide convergence results nor estimates on the speed of convergence. The articles
studying these questions either assume abstract conditions which are very difficult to check in
practice~\cite{Seneta1966,KnoblochPartzsch2010}, or prove exponential convergence for very weak
norms~\cite{CCLMMS09,Cattiaux2008,Littin2012}.

More probabilistic methods were also developed. The older reference is based on a renewal technique~\cite{Ferrari1995} and proves the
existence and convergence to a quasi-stationary distribution for discrete processes for which Assumption~(A1) is not satisfied. More
recently, one-dimensional birth and death processes with a unique quasi-stationary distribution have been shown to
satisfy~\eqref{eq:QLD} with uniform convergence in total variation~\cite{Martinez-Martin-Villemonais2012}. Convergence in total
variation for processes in discrete state space satisfying strong mixing conditions was obtained in~\cite{CloezThai2013} using Fleming-Viot particle
systems whose empirical distribution approximates conditional distributions~\cite{Villemonais2014}. 
Sufficient conditions for exponential convergence of conditioned systems in discrete time can be found in~\cite{DelMoral2013} with applications of discrete generation particle techniques in signal processing, statistical machine learning, and quantum physics. We also refer the reader to~\cite{DelMoral-Doucet2004,DelMoral-Guionnet2000,DelMoral-Miclo2000b} 
for approximations techniques of non absorbed trajectories in terms of genealogical trees.

In this work, we obtain in Section~\ref{sec:results-QSD} necessary and sufficient conditions for exponential convergence to a
unique quasi-stationary distribution for general (virtually any) Markov processes (we state a stronger form of Theorem~\ref{thm:QSD}). We also
obtain spectral properties of the infinitesimal generator as a corollary of our main result. 
Our non-spectral approach and results fundamentally differ from all the previously cited references,
except~\cite{Martinez-Martin-Villemonais2012,CloezThai2013} which only focus on very specific cases. In
Section~\ref{sec:results-Q-process}, we show, using penalisation techniques~\cite{Roynette_Vallois_Yor_2006}, that the same
conditions are sufficient to prove the existence of the $Q$-process and its exponential ergodicity, uniformly in total variation.
This is the first general result showing the link between quasi-stationary distributions and $Q$-processes, since we actually prove
that, for general Markov processes, the uniform exponential convergence to a quasi-stationary distribution implies the existence and
ergodicity of the $Q$-process.

Section~\ref{sec:application} is devoted to applications of the previous results to specific examples of processes. Our goal is not
to obtain the most general criteria, but to show how Assumption~(A) can be checked in different practical situations. We first obtain
necessary and sufficient conditions for one-dimensional birth and death processes with catastrophe in Section~\ref{sec:PNM-1d}. We
show next how the method of the proof can be extended to treat several multi-dimensional examples in
Section~\ref{sec:PNM-applications}. One of these examples is infinite-dimensional (as in~\cite{Collet2011a}) and assumes Brownian
mutations in a continuous type space. Our last example is the neutron transport process in a bounded domain, absorbed at the boundary
(Section~\ref{sec:diffusion}). This example belongs to the class of piecewise-deterministic Markov processes, for which up to our
knowledge no results on quasi-stationary distributions are known. In this case, the absorption rate is unbounded, in the sense that
the absorption time cannot be stochastically dominated by an exponential random variable with constant parameter. Other examples of
Markov processes with unbounded absorption rate can be studied thanks to Theorem~\ref{thm:QSD_full}. For example, the study of
diffusion processes on $\mathbb{R}^+$ absorbed at $0$, or on $\mathbb{R}_+^d$, absorbed at $0$ or
$\mathbb{R}_+^d\setminus(\mathbb{R}_+^*)^d$ is relevant for population dynamics (see for example~\cite{CCLMMS09,Cattiaux2008}) and is
studied in~\cite{champagnat-villemonais-14b}. More generally, the great diversity of applications of the very similar probabilistic criterion for
processes without absorption (see all the works building on~\cite{Meyn2009}) indicates the wide range of applications and 
extensions of our criteria that can be expected.

The paper ends with the proof of the main results of Section~\ref{sec:results-QSD} and~\ref{sec:results-Q-process} in
Sections~\ref{sec:QSD} and~\ref{sec:proofs-Q-process}.



\section{Existence and uniqueness of a quasi-stationary distribution}
\label{sec:results-QSD}

\subsection{Assumptions}
We begin with some comments on Assumption (A).

When $E$ is a Polish space, Assumption~(A1) implies that $X_t$ comes back fast in compact sets from any initial conditions. Indeed,
there exists a compact set $K$ of $E$ such that $\nu(K)>0$ and therefore, $\inf_{x\in E}\PP_x(\tau_{K\cup\{\partial\}}<t_0)>0$, where
$\tau_{K\cup\{\partial\}}$ is the first hitting time of $K\cup\{\partial\}$ by $X_t$. When $E=(0,+\infty)$ or $\mathbb{N}$ and
$\partial=0$, this is implied by the fact that the process $X$ \emph{comes down from infinity}~\cite{CCLMMS09} (see
Section~\ref{sec:discrete} for the discrete case). 

Assumption~(A2) means that the highest non-absorption probability among all initial points in $E$
has the same order of magnitude as the non-absorption probability starting from distribution
$\nu$. Note also that (A2) holds true when, for some $A\in {\cal E}$ such that $\nu(A)>0$ and some $c'_2>0$,
\begin{align*}
\inf_{y\in A}\PP_y\left(t<\tau_\partial\right)\geq c'_2 \sup_{x\in E}\PP_x\left(t<\tau_\partial\right).
\end{align*}

We now introduce the apparently weaker assumption (A') and the stronger assumption (A''), proved to be equivalent in Theorem~\ref{thm:QSD_full} below.
\paragraph{Assumption~(A')}	
	There exists a family of probability measures $(\nu_{x_1,x_2})_{x_1,x_2\in E}$ on $E$ such that,
\begin{itemize}
\item[(A$1'$)] there exists $t_0,c_1>0$ such that, for all $x_1,x_2\in E$,
  $$
  \PP_{x_i}(X_{t_0}\in\cdot\mid t_0<\tau_\partial)\geq c_1\nu_{x_1,x_2}(\cdot)\text{ for }i=1,2;
  $$
\item[(A$2'$)] there exist a constant $c_2>0$ such that for all $x_1,x_2\in E$ and $t\geq 0$,
  $$
  \PP_{\nu_{x_1,x_2}}(t<\tau_\partial)\geq c_2\sup_{x\in E}\PP_x(t<\tau_\partial).
  $$
\end{itemize}

\paragraph{Assumption~(A'')}	
Assumption (A1) is satisfied and
\begin{description}
\item[\textmd{(A$2''$)}] for any probability measure $\mu$ on $E$,
the constant $c_2(\mu)$ defined by
\begin{align*}
c_2(\mu):=  \inf_{t\geq0,\,\rho\in {\cal M}_1(E)}\frac{\PP_{\mu}(t<\tau_\partial)}{\PP_\rho(t<\tau_\partial)}
\end{align*}
is positive, where ${\cal M}_1(E)$ is the set of probability measures on $E$.
   \end{description}

\subsection{Results}

The next result is a detailed version of Theorem~\ref{thm:QSD}.

\begin{thm}
  \label{thm:QSD_full}
The following conditions \textup{(i)--(vi)} are equivalent.
\begin{description}
\item[\textmd{(i)}] Assumption \textup{(A)}.
\item[\textmd{(ii)}] Assumption \textup{(A')}.
\item[\textmd{(iii)}] Assumption \textup{(A'')}.
\item[\textmd{(iv)}] There exist a probability measure $\alpha$ on $E$ and two constants $C,\gamma>0$ such that, for all initial distribution
$\mu$ on $E$,
\begin{align}
  \label{eq:expo-cv}
  \left\|\PP_\mu(X_t\in\cdot\mid t<\tau_\partial)-\alpha(\cdot)\right\|_{TV}\leq C e^{-\gamma t},\ \forall t\geq 0.
\end{align}
\item[\textmd{(v)}] There exist a probability measure $\alpha$ on $E$ and two constants $C,\gamma>0$ such that, for all $x\in E$,
\begin{align*}
  \left\|\PP_x(X_t\in\cdot\mid t<\tau_\partial)-\alpha(\cdot)\right\|_{TV}\leq C e^{-\gamma t},\ \forall t\geq 0.
\end{align*}
\item[\textmd{(vi)}] There exists a probability measure $\alpha$ on $E$ such that
\begin{align}
\label{eq:epsilon-t}
\int_0^\infty  \sup_{x\in E}\left\|\PP_x(X_t\in\cdot\mid t<\tau_\partial)-\alpha(\cdot)\right\|_{TV}dt<\infty.
\end{align}
\end{description} 
  
  In this case, $\alpha$ is the unique quasi-stationary distribution for the process. In addition, if Assumption (A') is satisfied, then~\textup{(iv)} holds with the explicit bound
  \begin{align}
  \label{eq:expo-cv-explicit}
  \left\|\PP_\mu(X_t\in\cdot\mid t<\tau_\partial)-\alpha(\cdot)\right\|_{TV}\leq
  2(1-c_1c_2)^{\lfloor t/t_0\rfloor}.
  \end{align}
\end{thm}
\noindent This result and the others of this section are proved in Section~\ref{sec:QSD}.

One can expect that the constant $C$ in (iv) might depend on $\mu$ proportionally to $\|\mu-\alpha\|_{TV}$. This is indeed the case,
but with a constant of proportionality depending on $c_2(\mu)$, defined in Assumption~(A$2$''), as shown by the following result.
\begin{cor}
\label{cor:1}
Hypotheses (i--vi) imply that, for all probability measures $\mu_1,\mu_2$ on $E$, and for all $t>0$,
\begin{align*}
 \left\|\PP_{\mu_1}(X_t\in\cdot\mid t<\tau_\partial)-\PP_{\mu_2}(X_t\in\cdot\mid t<\tau_\partial)\right\|_{TV}\leq \frac{(1-c_1 c_2)^{\lfloor t/t_0\rfloor}}{c_2(\mu_1)\wedge c_2(\mu_2)}\|\mu_1-\mu_2\|_{TV}.
\end{align*}
\end{cor}

\begin{rem}
  \label{rem:explicit-constant-c_2}
  It immediately follows from~\eqref{eq:alphaPtbis} and (A'') that
  \begin{align}
  \label{eq:youpi2}
    e^{-\lambda_0 t}\leq  \sup_{\rho\in{\cal M}_1(E)}\P_{\rho}(t<\tau_\d) \leq   \frac{e^{-\lambda_0 t}}{c_2(\alpha)}.
  \end{align}
  In the proof of Theorem~\ref{thm:QSD_full}, we actually prove that one can take
  \begin{align}
    \label{eq:ineg-c2-de-alpha}
    c_2(\alpha)= \sup_{s>0} \exp\left(
      -\lambda_0s-\frac{Ce^{(\lambda_0-\gamma)s}}{1-e^{-\gamma s}}
    \right),
  \end{align}
  where $C$ and $\gamma$ satisfy~\eqref{eq:expo-cv}.
\end{rem}

\begin{rem}
  \label{rem:meyn-tweedie}
  In the case of Markov processes without absorption, Meyn and Tweedie~\cite[Chapter~16]{Meyn2009} give several equivalent criteria for the
  exponential ergodicity with respect to $\|\cdot\|_{TV}$, among which are unconditioned versions of (iv) and (v). The last results can be interpreted as an extension of these criteria to
  conditioned processes. Several differences remain.
  \begin{enumerate}
  \item In the case without absorption, the equivalence between 
   the unconditioned versions of criteria (iv) and (v)  is obvious. In our case, the proof is not immediate.
  \item In the case without absorption, the unconditioned version of  criterion (vi)  can be replaced by the weaker assumption $\sup_{x\in
      E}\|\P_x(X_t\in\cdot)-\alpha\|_{TV}\rightarrow 0$ when $t\rightarrow +\infty$. Whether (vi) can be improved in such a way
    remains an open problem in general. However, if one assumes that there exists $c'_2>0$ such that for all $t\geq 0$,
    \begin{align*}
      \inf_{\rho\in{\cal M}_1(E)}\P_\rho(t<\tau_\d)\geq c'_2 \sup_{\rho\in{\cal M}_1(E)}\P_\rho(t<\tau_\d),
    \end{align*}
    then one can adapt the arguments of Corollary~\ref{cor:1} to prove that
    \begin{align*}
      \sup_{x\in E}\left\|\PP_x(X_t\in\cdot\mid t<\tau_\partial)-\alpha(\cdot)\right\|_{TV}\xrightarrow[t\rightarrow\infty]{} 0
    \end{align*}
    implies (i--vi).
  \item The extension to quasi-stationary distributions of criteria based on Lyapunov functions as in~\cite{Meyn2009} requires a
    different approach because the survival probability and conditional expectations can not be expressed easily in terms of the
    infinitesimal generator.
  \item In the irreducible case, there is a weaker alternative to the Dobrushin-type criterion of Hypothesis~(A1) known as Doeblin's
    condition: there exist $\mu\in{\cal M}_1(E)$, $\varepsilon<1$, $t_0,\delta>0$ such that, for all measurable set $A$ satisfying
    $\mu(A)>\varepsilon$,
    \begin{align*}
      \inf_{x\in E}\P_x(X_{t_0}\in A)\geq \delta.
    \end{align*}
    It is possible to check that the conditional version of this criterion implies the existence of a probability measure $\nu\neq \mu$ such that  (A1) is satisfied.
    Unfortunately $\nu$ is far from being explicit in this case and (A2), which must involve the measure $\nu$, is no more a tractable condition, unless one can prove directly (A$2$'') instead of (A2).
  \end{enumerate}
\end{rem}

\medskip \noindent It is well known (see~\cite{MV12}) that when $\alpha$ is a quasi-stationary distribution, there exists $\lambda_0>0$ such that, for all $t\geq 0$,
\begin{align}
\label{eq:alphaPt}
\PP_{\alpha}(t<\tau_{\partial})=e^{-\lambda_0 t}\quad\text{and}\quad e^{\lambda_0 t}\alpha P_t=\alpha.
\end{align}
The next result is a detailed version of Proposition~\ref{prop:2intro}.
\begin{prop}
\label{prop:2}
There exists a non-negative function $\eta$ on $E\cup\{\d\}$, positive on $E$ and vanishing on $\d$, defined by
\begin{align*}
\eta(x)=\lim_{t\rightarrow\infty} \frac{\P_x(t<\tau_\d)}{\P_\alpha(t<\tau_\d)}=\lim_{t\rightarrow +\infty} e^{\lambda_0 t}\P_x(t<\tau_\d),
\end{align*}
where the convergence holds for the uniform norm on $E\cup\{\d\}$ and $\alpha(\eta)=1$. Moreover, the function $\eta$ is bounded, belongs to the domain of the infinitesimal generator $L$ of the semi-group $(P_t)_{t\geq 0}$ on $({\cal B}(E\cup\{\d\}),\|\cdot\|_\infty)$ and
\begin{align*}
L\eta=-\lambda_0\eta.
\end{align*}
\end{prop}

\noindent In the irreducible case, exponential ergodicity is known to be related to a spectral gap property (see for instance~\cite{ledoux2005concentration}). Our results imply a similar property under the assumptions (i--vi) for the infinitesimal generator $L$ of the semi-group on $({\cal B}(E\cup\{\d\}),\|\cdot\|_{\infty})$.

\begin{cor}
\label{cor:2}
If $f\in {\cal B}(E\cup\{\d\})$ is a right eigenfunction for $L$ for an eigenvalue $\lambda$, then either  
\begin{enumerate}[itemsep=0pt]
\item $\lambda=0$ and $f$ is constant,
\item or $\lambda=-\lambda_0$
and $f=\alpha(f)\eta$,
\item or $\lambda\leq -\lambda_0-\gamma$, $\alpha(f)=0$ and $f(\d)=0$.
\end{enumerate}
\end{cor}

\section{Existence and exponential ergodicity of the $Q$-process}
\label{sec:results-Q-process}
We now study the behavior of
the $Q$-process. The next result is a detailed version of Theorem~\ref{thm:Q-processintro}.

\begin{thm}
\label{thm:Q-process}
Assumption~(A) implies the three following properties.
\begin{description}
\item[\textmd{(i) Existence of the $Q$-process.}] There exists a family $(\QQ_x)_{x\in E}$ of
probability measures on $\Omega$ defined by
$$
\lim_{t\rightarrow+\infty}\PP_x(A\mid t<\tau_\partial)=\QQ_x(A)
$$
for all ${\cal F}_s$-measurable set $A$.
The process $(\Omega,({\cal F}_t)_{t\geq 0},(X_t)_{t\geq
  0},(\QQ_x)_{x\in E})$ is an $E$-valued homogeneous Markov process. In addition, if $X$ is a strong Markov process under $\P$, then so is $X$ under $\Q$.
\item[\textmd{(ii) Transition kernel.}] The transition kernel of the Markov process $X$ under $(\QQ_x)_{x\in E}$ is given by
\begin{align*}
\tilde{p}(x;t,dy)=e^{\lambda_0 t}\frac{\eta(y)}{\eta(x)}p(x;t,dy).
\end{align*}
 In other words, for all $\varphi\in{\cal B}(E)$ and $t\geq 0$,
\begin{align}
\label{eq:semi-group-Q}
\tilde{P}_t\varphi(x)=\frac{e^{\lambda_0 t}}{\eta(x)}P_t(\eta\varphi)(x)
\end{align}
where $(\tilde{P}_t)_{t\geq 0}$ is the semi-group of $X$ under $\Q$.
\item[\textmd{(iii) Exponential ergodicity.}] The probability measure $\beta$ on $E$ defined by
\begin{align*}
\beta(dx)
=\eta(x)\alpha(dx).
\end{align*}
is the unique invariant distribution of $X$ under $\QQ$. Moreover, for any initial distributions $\mu_1,\mu_2$ on $E$,
\begin{align*}
\left\|\Q_{\mu_1}(X_t\in\cdot)-\Q_{\mu_2}(X_t\in\cdot)\right\|_{TV}\leq (1-c_1c_2)^{\lfloor t/t_0\rfloor}\|\mu_1-\mu_2\|_{TV},
\end{align*}
where $\Q_\mu=\int_E \Q_x\,\mu(dx)$.
\end{description}
\end{thm}

\noindent Note that, as an immediate consequence of Theorem~\ref{thm:QSD_full}, the uniform exponential convergence to a
quasi-stationary distribution implies points (i--iii) of Theorem~\ref{thm:Q-process}.


We investigate now the characterization of the $Q$-process in term of its weak infinitesimal generator  (see~\cite[Ch I.6]{Dynkin1965}).
 Let us recall the definition of the bounded pointwise convergence: for all $f_n$, $f$ in ${\cal B}(E\cup\{\d\})$, we say that
\begin{align*}
\text{b.p.-}\lim_{n\rightarrow\infty} f_n=f
\end{align*}
if and only if $\sup_{n} \|f_n\|_{\infty}<\infty$ and for all $x\in E\cup\{\d\}$, $f_n(x)\rightarrow f(x)$.

The weak infinitesimal generator $L^w$ of $(P_t)$ is defined as
\begin{align*}
L^w f= \text{b.p.-}\lim_{h\rightarrow 0} \frac{P_h f-f}{h},
\end{align*}
for all $f\in {\cal B}(E\cup\{\d\})$ such that the above b.p.--limit exists and
\begin{align*}
\text{b.p.-}\lim_{h\rightarrow 0} P_h L^w f=L^w f.
\end{align*}
We call weak domain and denote by ${\cal D}(L^w)$ the set of such functions $f$.
We define similarly the b.p.--limit in ${\cal B}(E)$ and the weak infinitesimal generator $\tilde{L}^w$ of $(\tilde{P}_t)$ and its weak domain ${\cal D}(\tilde{L}^w)$.

\begin{thm}
\label{thm:generator-of-Q-process}
Assume that (A) is satisfied. 
Then 
\begin{align}
\label{eq:domain-l-tilde}
{\cal D}(\tilde{L}^w)=\left\{f\in{\cal B}(E),\;\eta f\in{\cal D}(L^w)\text{ and }\frac{L^w(\eta f)}{\eta}\text{is bounded}\right\}
\end{align}
 and, for all $f\in{\cal D}(\tilde{L}^w)$,
\begin{align*}
\tilde{L}^wf=\lambda_0 f+\frac{L^w(\eta f)}{\eta}.
\end{align*}
If in addition $E$ is a topological space and $\cal E$ is the Borel $\sigma$-field, and if for all open set $U\subset E$ and $x\in U$,
\begin{align}
\label{eq:stochasticallycontinuous}
\lim_{h\rightarrow 0} p(x;h,U)
= \lim_{h\rightarrow 0} P_h \11_{U}(x)=1,
\end{align}
then the semi-group $(\tilde{P}_t)$ is uniquely determined by its weak infinitesimal generator $\tilde{L}^w$.
\end{thm}

Let us emphasize that \eqref{eq:stochasticallycontinuous} is obviously satisfied if the process $X$ is almost surely c\`adl\`ag.

\begin{rem}
One can wonder if the weak infinitesimal generator can be replaced in the previous result by the standard one. Then Hille-Yoshida Theorem would give necessary and sufficient condition for  a strongly continuous contraction semi-group on a Banach space $B$ to be characterized by its standard infinitesimal generator (see for example \cite[Thm 1.2.6, Prop 1.2.9]{Ethier1986}). However this is an open question that we couldn't solve. To understand the difficulty, observe that even the strong continuity of $\tilde{P}$ cannot be easily deduced from the strong continuity of $P$: in view of~\eqref{eq:semi-group-Q}, if $\eta f\in B$, we have
\begin{align*}
\left\|\tilde{P}_t f-f\right\|_\infty\xrightarrow[t\rightarrow 0]{} 0
\quad\Leftrightarrow\quad 
\left\|\frac{1}{\eta}\left(P_t(\eta f)-\eta f\right)\right\|_\infty\xrightarrow[t\rightarrow 0]{} 0.
\end{align*}
We don't know whether the last convergence can be deduced from the strong continuity of $P$ or if counter examples exist.

%
%
%

\end{rem}

\section{Applications}
\label{sec:application}

This section is devoted to the application of Theorems~\ref{thm:QSD_full} and~\ref{thm:Q-process} to discrete and continuous
examples. Our goal is to show how Assumption~(A) can be checked in different practical situations.

\subsection{Generalized birth and death processes}
\label{sec:discrete}

Our goal is to apply our results to generalized birth and death processes. In subsection~\ref{sec:PNM-1d}, we extend known criteria to one
dimensional birth and death processes with catastrophe. In subsection~\ref{sec:PNM-applications}, we apply a similar method to multi-dimensional and infinite dimensional birth and death processes.

\subsubsection{Birth and death processes with catastrophe}
\label{sec:PNM-1d}
We consider an extension of classical birth and death processes with possible mass extinction. Our goal is to extend the recent
result from~\cite{Martinez-Martin-Villemonais2012} on the characterisation of exponential convergence to a unique quasi-stationary
distribution. The existence of quasi-stationary distributions for similar processes was studied in~\cite{vanDoorn2012}.

Let $X$ be a birth and death process on $\mathbb{Z}_+$ with birth rates $(b_n)_{n\geq 0}$ and death rates $(d_n)_{n\geq 0}$ with
$b_0=d_0=0$ and $b_k,d_k>0$ for all $k\geq 1$. We also allow the process to jump to $0$ from any state $n\geq 1$ at rate $a_n\geq 0$.  In particular, the jump rate from $1$ to $0$ is $a_1+d_1$. This process is absorbed in $\d=0$. 

\begin{thm}
\label{pro:bd-proc-1d} 
Assume that $\sup_{n\geq 1} {a_n}<\infty$. Conditions {\upshape (i-vi)} of Theorem~\ref{thm:QSD_full} are equivalent to
\begin{align}
\label{eq:def-S}
S:=\sum_{k\geq 1}\frac{1}{d_k\alpha_k}\sum_{l\geq k} \alpha_l <\infty,
\end{align}
with
$
    \alpha_k=\left(\prod_{i=1}^{k-1} b_i\right)/\left(\prod_{i=1}^{k} d_i\right).
$

Moreover, there exist constants $C,\gamma>0$ such that
  \begin{align}
  \label{eq:bd-1}
    \left\|\PP_{\mu_1}(X_t\in\cdot\mid t<\tau_\partial)-\PP_{\mu_2}(X_t\in\cdot\mid t<\tau_\partial)\right\|_{TV}\leq Ce^{-\gamma
      t}\|\mu_1-\mu_2\|_{TV}
  \end{align}
  for all $\mu_1,\mu_2\in\mathcal{M}_1(E)$ and $t\geq 0$.
\end{thm}

The last inequality and the following corollary of Theorem~\ref{thm:Q-process} are original results, even in the simpler case of birth and death processes without catastrophes.

\begin{cor}
Under the assumption that $\sup_{n\geq 1} a_n<\infty$ and $S<\infty$, the family $(\QQ_x)_{x\in E}$ of
probability measures on $\Omega$ defined by
\begin{align}
  \label{eq:bd-2}
\lim_{t\rightarrow+\infty}\PP_x(A\mid t<\tau_\partial)=\QQ_x(A),\ \forall A\in{\cal F}_s,\ \forall s\geq 0,
\end{align}
is well defined.
In addition, the process $X$ under $(\Q_x)$ admits the unique invariant distribution
\begin{align*}
\beta(dx)=\eta(x)\alpha(dx)
\end{align*}
and there exist constants $C,\gamma>0$ such that, for any $x\in E$,
\begin{align}
  \label{eq:bd-3}
\left\|\Q_{x}(X_t\in\cdot)-\beta\right\|_{TV}\leq Ce^{-\gamma t}.
\end{align}
\end{cor}

%

\begin{rem}
  \label{rem:PNM-1d}
  In view of Point 2.\ in Remark~\ref{rem:meyn-tweedie}, we actually also have the following property. Conditionally on non-extinction, $X$ converges
  uniformly in total variation to some probability measure $\alpha$ if and only if it satisfies (i--vi).
\end{rem}


\begin{proof}[Proof of Theorem~\ref{pro:bd-proc-1d}]
Let $Y$ be the birth and death process on $\Z_+$ (without catastrophe) with birth and death rates $b_n$ and $d_n$ from state $n$. The process $X$ and $Y$ can be coupled such that $X_t=Y_t$, for all $t<\tau_\d$ almost surely.

We recall (see~\cite{vanDoorn1991}) that $S<\infty$ if and only if the birth and death process $Y$ comes down from infinity, in the sense that 
\begin{align*}
\sup_{n\geq 0}\E_n(\tau'_\d)<\infty,
\end{align*}
where $\tau'_\d=\inf\{t\geq 0,\ Y_t=0\}$.
More precisely, for any $z\geq 0$, 
\begin{align}
\label{eq:descente-infini}
\sup_{n\geq z}\E_n(T'_z)= \sum_{k\geq z+1}\frac{1}{d_k\alpha_k}\sum_{l\geq k} \alpha_l <\infty,
\end{align}
where $T'_z=\inf\{t\geq 0,\ Y_t\leq z\}$. Note that this equality remains true even when the sum is infinite.
\medskip

Let us first assume that (A1) is satisfied. This will be sufficient to prove that $S<\infty$. Let $z$ be large enough so that
$\nu(\{1,\ldots,z\})>0$. Then, for all $n\geq 1$,
\begin{align*}
\P_n(Y_{t_0}\leq z)&\geq \P_n(X_{t_0}\leq z\text{\ and\ }t_0\leq\tau_\d)\\
&\geq c_1 \nu(\{1,\ldots,z\})\P_n(t_0\leq\tau_\d).
\end{align*}
Since the jump rate of $X$ to $0$ from any state is always smaller than $\overline{q}=d_1+\sup_{n\geq 1} {a_n}$, the absorption time dominates an exponential r.v.\ of
parameter $\overline{q}$. Thus $\P_n(t_0\leq\tau_\d)\geq e^{-\overline{q} t_0}$ and hence 
\begin{equation*}
  \inf_{n\geq 1}\P_n(Y_{t_0}\leq z)\geq ce^{-\overline{q} t_0},  
\end{equation*}
for some $c>0$. Defining $\theta=\inf\{n\geq 0,\ Y_{nt_0}\leq z\}$, we deduce from the Markov property that, for all $n\geq 1$ and
$k\geq 0$,
$$
\P_n(\theta> k+1\mid\theta> k)\leq\sup_{m>z}\P_m(Y_{t_0}\geq z)\leq 1-ce^{-\overline{q} t_0}.
$$
Thus $\P_n(\theta> k)\leq(1-ce^{-\overline{q} t_0})^k$ and $\sup_{n\geq z}\E_n(T'_z)\leq \sup_{n\geq z}\E_n(t_0\theta)<\infty$.
By~\eqref{eq:descente-infini}, this entails $S<\infty$. \medskip

Conversely, let us assume that $S<\infty$. 
%
For
all $\varepsilon>0$ there exists $z$ such that
\begin{align*}
\sup_{n\geq z}\E_n(T'_z)\leq \varepsilon.
\end{align*}
Therefore, $\sup_{n\geq z}\P_n(T'_z\geq 1)\leq\varepsilon$ and, applying recursively the Markov property, $\sup_{n\geq z}\P_n(T'_z\geq
k)\leq\varepsilon^k$.
Then, for all $\lambda>0$, there exists $z\geq 1$ such that
\begin{equation}
  \label{eq:moment-expo}
  \sup_{n\geq 1}\E_n(e^{\lambda T'_z})<+\infty.  
\end{equation}

Fix $x_0\in E$ and let us check that this exponential moment implies (A2) and then (A1) with $\nu=\delta_{x_0}$. We choose $\lambda=1+\overline{q}$ and apply the previous construction of $z$. Defining the finite set $K=\{1,2,\ldots,z\}\cup\{x_0\}$ and $\tau_K=\inf\{t\geq 0,\,X_t\in K\}$, we thus have
\begin{align}
\label{eq:expo-moment}
A:=\sup_{x\in E}\E_x(e^{\lambda \tau_K\wedge \tau_\d})<\infty.
\end{align}

 Let us
  first observe that for all $y,z\in K$, $\P_y(X_1=z)\P_z(t<\tau_\d)\leq\P_y(t+1<\tau_\d)\leq\P_y(t<\tau_\d)$. Therefore, the
  constant $C^{-1}:=\inf_{y,z\in K}\P_y(X_1=z)>0$ satisfies the following inequality:
  \begin{equation}
    \label{eq:ineq-PB-catastrophe}
    \sup_{x\in K}\P_x(t<\tau_\d)\leq C\inf_{x\in K}\P_x(t<\tau_\d),\quad\forall t\geq 0.
  \end{equation}
Moreover, since $\lambda$ is larger than the maximum absorption rate $\overline{q}$, for $t\geq s$,
\begin{align*}
e^{-\lambda s}\P_{x_0}(t-s<\tau_\d)\leq \P_{x_0}(t-s<\tau_\d)\inf_{x\geq 1}\P_{x}(s<\tau_\d)\leq \P_{x_0}(t<\tau_\d).
\end{align*}
For all $x\in E$, we deduce from Chebyshev's inequality and~\eqref{eq:expo-moment} that 
\begin{align*}
\P_x(t<\tau_K\wedge \tau_\d)\leq Ae^{-\lambda t}.
\end{align*}
 Using the last three inequalities and the strong Markov property, we have
\begin{align*}
    \P_x(t<\tau_\d) & =\P_x(t<\tau_K\wedge\tau_\d)+\P_x(\tau_K\wedge\tau_\d\leq t<\tau_\d)\\
&\leq Ae^{-\lambda t}+\int_0^t \sup_{y\in K\cup\{\d\}}\P_y(t-s<\tau_\d)\P_x(\tau_K\wedge\tau_\d\in ds) \\
        & \leq A\P_{x_0}(t<\tau_\d)+ C\int_0^t \P_{x_0}(t-s<\tau_\d)\P_x(\tau_K\wedge\tau_\d\in ds) \\
    & \leq  A\P_{x_0}(t<\tau_\d)+C\,\P_{x_0}(t<\tau_\d)\int_0^t e^{\lambda s}\,\P_x(\tau_K\wedge\tau_\d\in ds)\\
    &\leq A(1+C)\P_{x_0}(t<\tau_\d).
  \end{align*}
This shows (A2) for $\nu=\delta_{x_0}$. 

Let us now show that $(A1)$ is satisfied. We have, for all $x\in E$,
\begin{align*}
\P_x(\tau_K<t)&= \P_x(\tau_K<t\wedge\tau_\d)\geq \P_x(t<\tau_\d)-\P_x(t<\tau_K\wedge\tau_\d)\\
&\geq e^{-\overline{q}t}-Ae^{-\lambda t}.
\end{align*}
Since $\lambda>\overline{q}$, there exists $t_0>0$  such that
\begin{align*}
\inf_{x\in E} \P_x(\tau_K<t_0-1) >0.
\end{align*}
But the irreducibility of $X$ and the finiteness of $K$ imply that $\inf_{y\in K}\P_y(X_1=x_0)>0$, thus the Markov property entails
\begin{align*}
\inf_{x\in E}\P_x(X_{t_0}=x_0)\geq\inf_{x\in E}\mathbb{E}_x[\mathds{1}_{\tau_K<t_0-1}\inf_{y\in K}\mathbb{P}_y(X_1=x_0)e^{-q_{x_0}(t_0-1-\tau_K)}]>0,
\end{align*}
where $q_{x_0}=a_{x_0}+b_{x_0}+d_{x_0}$ is the the jump rate from state $x_0$, which implies (A1) for $\nu=\delta_{x_0}$. Finally, using Theorem~\ref{thm:QSD_full}, we have proved that (i--vi) holds.

In order to conclude the proof, we use Corollary~\ref{cor:1} and the fact that
$$
\inf_{x\in E} c_2(\delta_x)\geq  \inf_{x\in E}\P_x(X_{t_0}=x_0)c_2(\delta_{x_0})>0.
$$
This justifies the last part of Theorem~\ref{pro:bd-proc-1d}.

%
\end{proof}

\subsubsection{Extensions to multi-dimensional birth and death processes}
\label{sec:PNM-applications}

In this section, our goal is to illustrate how the previous result and proof apply in various multi-dimensional cases, using
comparison arguments. We focus here on a few instructive examples in order to illustrate the tools of our method and its
applicability to a wider range of models. We will consider three models of multi-specific populations, with competitive or cooperative
interaction within and among species.

Our first example deals with a birth and death process in $\Z_+^d$, where each coordinate represent the number of individuals of distinct types (or in different geographical patches). We will assume that mutations (or
migration) from each type (or patch) to each other is possible at birth, or during the life of individuals. In this example, the absorbing state $\d=0$ corresponds to the total extinction of the population.

We consider in our second example a cooperative birth and death process without mutation (or migration), where extinct types remain
extinct forever. In this case, the absorbing states are $\d=\mathbb{Z}_+^d\setminus\mathbb{N}^d$, where $\mathbb{N}=\{1,2,\ldots\}$.

Our last example shows how these techniques apply to discrete populations with continuous type space and Brownian genetical drift.
Such multitype birth and death processes in continuous type space naturally arise in evolutionary
biology~\cite{ChampagnatMeleard2007} and the existence of a quasi-stationary distribution of similar processes has been studied
in~\cite{Collet2011a}.

\bigskip

\begin{exa}[Birth and death processes with mutation or migration]$\ $\\
\label{sec:PNM-with-mutation}\noindent We consider a $d$-dimensional birth and death process with  type-dependent individual birth and death rates $X$, where individuals compete with each others with type dependent coefficients. 
We denote by $\lambda_{ij}> 0$ the birth rate of an individual of type $j$ from an individual of type $i$, $\mu_i> 0$ the death rate of an individual of type $i$ and by $c_{ij}>0$ the death rate of an individual of type $i$ from competition with an individual of type $j$. More precisely, if $x\in\Z_+^d$, denoting by $b^i(x)$ (resp. $d^i(x)$) the birth (resp. death) rate of an individual of type $i$ in the population $x$,
we have
\begin{align*}
\begin{cases}
b^i(x)=\sum_{j=1}^d \lambda_{ji}x_j, \\
d^i(x)=\mu_i x_i + \sum_{j=1}^d c_{ij} x_ix_j.
\end{cases}
\end{align*}
Note that $\d=0$ is the only absorbing state for this process.


The process $X$ can be coupled with a birth and death process $Y$ such that $|X_t|\leq Y_t$ with birth and death rates
\begin{gather}
  b_n:=nd\sup_{i,j}\lambda_{ij} \geq \sup_{x\in\mathbb{Z}_+^d,\ |x|=n} \sum_{i=1}^d b^i(x) \label{eq:def-bn}\\
  d_n:=n\inf_{i}\mu_i+n^2\inf_{i,j}c_{ij}\leq \inf_{x\in\mathbb{Z}_+^d,\ |x|=n} \sum_{i=1}^d d^i(x), \label{eq:def-dn}
\end{gather}
where $|x|=x_1+\cdots+x_d$. We can check that $S$, defined in~\eqref{eq:def-S}, is finite and hence one obtains~\eqref{eq:moment-expo} and \eqref{eq:expo-moment} exactly as in the proof of Theorem~\ref{pro:bd-proc-1d}. From these inequalities, the proof of (A2) and (A1) is the same as for Theorem~\ref{pro:bd-proc-1d}, with $K=\{x\in E,\,|x|\leq z\}\cup\{x_0\}$ and $\overline{q}=\max_{i}\mu_i+c_{ii}<\infty$.

Hence there exists a unique quasi-stationary distribution. Moreover, conditions (i--vi) hold as well as~\eqref{eq:bd-1},~\eqref{eq:bd-2} and~\eqref{eq:bd-3}.

%

\end{exa}

\bigskip
\begin{exa}[Weak cooperative birth and death process without mutation]$\ $\\
\label{sec:PNM-without-mutation}\noindent We consider a $d$-dimensional birth and death process with  type-dependent individual birth and death rates, where individuals of the same type compete with each others and where individuals of different types cooperate with each others. 
Denoting by $\lambda_i\geq 0$ and $\mu_i\geq 0$ the individual birth and death rates, by $c_{ii}>0$ the intra-type individual competition rate and by $c_{ij}\geq 0$ the inter-type individual cooperation rate for any types $i\neq j$, we thus have the following multi-dimensional birth and death rates:
\begin{align*}
\begin{cases}
b^i(x)=\lambda_i x_i + \sum_{j\neq i} c_{ij} x_ix_j\\
d^i(x)=\mu_i x_i + c_{ii} x_i^2.
\end{cases}
\end{align*}
This process is absorbed in
$\d=\mathbb{Z}_+^d\setminus\mathbb{N}^d$.

We assume that the inter-type cooperation is weak, relatively to the competition between individuals of same types. More formally, we assume that, for all $i\in \{1,\ldots,d\}$,
\begin{align}
\label{eq:weak-cooperation-assumption}
\left(1-\frac{1}{d}\right)\max_{i\neq j} \frac{c_{ij}+c_{ji}}{2} < \frac{1}{\beta}
\end{align}
where $\beta=\sum_{j=1}^d \frac{1}{c_{jj}}$. 

We claim that there exists a unique quasi-stationary distribution and that conditions (i--vi) hold as well as~\eqref{eq:bd-1},~\eqref{eq:bd-2} and~\eqref{eq:bd-3}.

Indeed the same coupling argument as in the last example can be used with $b_n$ and $d_n$ defined as
\begin{align*}
b_n&:= n\max_{i\in\{1,\ldots,d\}} \lambda_i +  n^2\left(1-\frac{1}{d}\right) \max_{i< j}  \frac{c_{ij}+c_{ji}}{2},\\
d_n&:= n \min_{i\in\{1,\ldots,d\}}\mu_i + \frac{n^2}{\beta}.
\end{align*}
From this definition, one can check that 
\begin{align*}
\sup_{x\in\Z_+^d,\,|x|=n}&\sum_{i=1}^d b^i(x)\leq 
\max_{i\in\{1,\ldots,d\}} \lambda_i \sum_{i=1}^d x_i +  \max_{i< j}  \frac{c_{ij}+c_{ji}}{2}\sum_{i\neq j}x_ix_j\\
&\leq n \max_{i\in\{1,\ldots,d\}} \lambda_i+ \max_{i< j}  \frac{c_{ij}+c_{ji}}{2} \left(n^2-n^2\min_{y\in \R_+^d, |y|=1} \sum_{i=1}^d y_i^2\right)= b_n
\end{align*}
and
\begin{align*}
\inf_{x\in\Z_+^d,\,|x|=n}\sum_{i=1}^d d^i(x) 
   \geq n \min_{i\in\{1,\ldots,d\}}\mu_i + n^2 \min_{y\in \R_+^d, |y|=1} \sum_{i=1}^d c_{ii}y_i^2.
\end{align*}
Since the function $y\mapsto \sum_{i=1}^d c_{ii}y_i^2$ on $\{y\in \R_+^d, |y|=1\}$ reaches its minimum at $(1/c_{11},\ldots,1/c_{dd})/\beta$, we have
\begin{align*}
\inf_{x\in\Z_+^d,\,|x|=n}\sum_{i=1}^d d^i(x) \geq d_n.
\end{align*}
Now we deduce from~\eqref{eq:weak-cooperation-assumption} that 
$b_n/d_{n+1}$ converges to a limit smaller than 1. Hence,
\begin{align*}
    S':=\sum_{k\geq d}\frac{1}{d_k\alpha'_k}\sum_{l\geq k} \alpha'_l <\infty,
  \end{align*}
  with $\alpha'_k=\left(\prod_{i=d}^{k-1} b_i\right)/\left(\prod_{i=d}^{k} d_i\right)$. This implies that, for any $\lambda>0$, there exists $z$ such that~\eqref{eq:moment-expo} holds. Because of the cooperative assumption, the maximum absorption rate is given by $\overline{q}=\max_{i} \mu_i+c_{ii}<\infty$ and we can conclude following the proof of Theorem~\ref{pro:bd-proc-1d}, as in Example~\ref{sec:PNM-with-mutation}.



\end{exa}

\begin{exa}[Infinite dimensional birth and death processes]$\ $\\
\label{sec:PNM-infini-dim}We consider a birth and death process $X$ evolving in the set ${\cal M}$ of finite point measures on $\mathbb{T}$, where $\mathbb{T}$ is the set of individual types. We refer to~\cite{Collet2011a} for a study of the existence of quasi-stationary distributions for similar processes.

If $X_t=\sum_{i=1}^n \delta_{x_i}$, the population at time $t$ is composed of $n$ individuals of types $x_1,\ldots,x_n$. For simplicity, we assume that $\mathbb{T}$ is the unit torus of $\R^d$, $d\geq 1$, and that each individual's type is subject to mutation during its life according to independent standard Brownian motions in $\mathbb{T}$.

We denote by $\lambda(x)> 0$ the birth rate of an individual of type $x\in\mathbb{T}$, $\mu(x)> 0$ its death rate and by $c(x,y)>0$ the death rate of an individual of type $x$ from competition with an individual of type $y$, where $\lambda$, $\mu$ and $c$ are continuous functions. More precisely, if $\xi=\sum_{i=1}^n \delta_{x_i}\in{\cal M}\setminus\{0\}$, denoting by $b(x_i,\xi)$ (resp. $d(x_i,\xi)$) the birth (resp. death) rate of an individual of type $x_i$ in the population $\xi$,
we have
\begin{align*}
\begin{cases}
b(x_i,\xi)=\lambda(x_i), \\
d(x_i,\xi)=\mu(x_i) + \int_{\mathbb{T}}c(x_i,y)d\xi(y).
\end{cases}
\end{align*}
This corresponds to clonal reproduction.
Similarly to Example~\ref{sec:PNM-with-mutation}, we assume that the process is absorbed at $\d=0$ (see~\cite{ChampagnatMeleard2007}
for the construction of this process).

We claim that there exists a unique quasi-stationary distribution and that conditions (i--vi) hold as well as~\eqref{eq:bd-1},~\eqref{eq:bd-2} and~\eqref{eq:bd-3}.

Indeed, one can check as in example~1 that there exists $t_0\geq 0$ such that
\begin{align*}
\inf_{\xi\in{\cal M}}\P_{\xi}(|\text{Supp }X_{t_0}|=1)>0.
\end{align*}
Observe that for all measurable set $\Gamma\subset \mathbb{T}$,
\begin{align*}
\P_{\delta_x}(X_{1}=\delta_y,\text{ with }y\in \Gamma)\geq e^{-\sup_{x\in\mathbb{T}}(b(x,\delta_x)+d(x,\delta_x))}\P_x(\tilde{B}_1\in \Gamma),
\end{align*}
where $\tilde{B}$ is a standard Brownian motion in $\mathbb{T}$. Hence, defining $\nu$ as the law of $\delta_U$, where $U$ is uniform on $\mathbb{T}$, there exists $c_1>0$ such that, for all measurable set $A\subset {\cal M}$,
\begin{align*}
\inf_{\xi\in{\cal M}}\P_{\xi}(X_{t_0+1}\in A)\geq c_1\nu(A).
\end{align*}
This entails (A1) for the measure $\nu$. As in the two previous examples, (A2) follows from similar computations as in the proof of
Theorem~\ref{pro:bd-proc-1d}. In particular, there exists $n_0$ such that
$\sup_{\xi\in\mathcal{M}}\E_\xi(e^{\lambda\tau_{K_{n_0}}\wedge\tau_\d})<\infty$, where $\lambda=1+\sup_{x\in\mathbb{T}}[\mu(x)+c(x,x)]$
and
$$
K_{n_0}=\{\xi\in\mathcal{M},\ |\text{Supp}(\xi)|\leq n_0\}.
$$
The new difficulty is to prove~\eqref{eq:ineq-PB-catastrophe}. Since absorption occurs only from states with one individual, this is equivalent to: there exists a constant $C>0$ such
that, for all $t\geq 0$, $x_0\in\mathbb{T}$,
\begin{equation}
  \label{eq:proof-brownian-mutation-1}
  \P_{\delta_{x_0}}(t<\tau_\d)\geq C\sup_{\xi\in K_{n_0}}\P_\xi(t<\tau_\d).  
\end{equation}
If this holds, we conclude the proof as for Theorem~\ref{pro:bd-proc-1d}.

To prove~\eqref{eq:proof-brownian-mutation-1}, let us first observe that the jump rate from any state of $K_{n_0}$ is uniformly
bounded from above by a constant $\rho<\infty$. Hence, we can couple the process $X_t$ with an exponential r.v.\ $\tau$ with
parameter $\rho$ independent of the Brownian motions driving the mutations in such a way that $X$ does not jump in the time interval
$[0,\tau]$. For any $x\in\mathbb{T}^n$ and any Brownian motion $B$ on $\mathbb{T}^n$ independant of $\tau$, we have
$$
\P(x+B_{\tau\wedge 1}\in\Gamma)\leq \P(x+B_{\tau}\in\Gamma)+\P(x+B_{1}\in\Gamma)\leq C\text{Leb}(\Gamma),\quad\forall \Gamma\in\mathcal{B}(\mathbb{T}^n),
$$
where the last inequality follows from the explicit density of $B_{\tau}$~\cite[Eq.\,1.0.5]{borodin-salminen-02}. From this it is
easy to deduce that there exists $C'<\infty$ such that, for all $1\leq n\leq n_0$, $A\in\mathcal{B}(K_n\setminus K_{n-1})$ and
$\xi\in K_n\setminus K_{n-1}$,
$$
\P_\xi(X_{\tau\wedge 1}\in A)\leq C'\mathcal{U}_n(A),
$$
where $\mathcal{U}_n$ is the law of $\sum_{i=1}^n\delta_{U_i}$, where $U_1,\ldots,U_n$ are i.i.d. uniform r.v. on $\mathbb{T}$. Since
one also has
$$
\P_{\delta_{x_0}}(X_1\in A)\geq\mathbb{P}_{\delta_{x_0}}(|\text{Supp}(X_{1/2})|=n,\, X_1\in A)\geq C''\mathcal{U}_n(A),\quad\forall A\in\mathcal{B}(K_n\setminus K_{n-1})
$$
for a constant $C''$ independent of $A$, we have proved that
$$
\P_{\delta_{x_0}}(X_1\in A)\geq C \sup_{\xi\in K_n\setminus K_{n-1}}\P_\xi(X_{\tau\wedge 1}\in A),\quad\forall A\in\mathcal{B}(K_n\setminus K_{n-1})
$$
for a constant $C$ independent of $A$ and $n\leq n_0$. We can now prove~\eqref{eq:proof-brownian-mutation-1}: for all fixed $\xi\in
K_n\setminus K_{n-1}$,
\begin{align*}
  \P_{\delta_{x_0}}(t+1<\tau_\d) &\geq \int_{K_n\setminus K_{n-1}}\P_{\zeta}(t<\tau_\d)\P_{\delta_{x_0}}(X_1\in d\zeta) \\
  & \geq C\int_{K_n\setminus K_{n-1}}\P_{\zeta}(t<\tau_\d)\P_\xi(X_{\tau\wedge 1}\in d\zeta) \\
  & = C\P_\xi(\tau\wedge 1+t<\tau_\d)\geq C\P_\xi(t+1<\tau_\d).
\end{align*}

\end{exa}

\subsection{Absorbed neutron transport process}
\label{sec:diffusion}

The propagation of neutrons in fissible media is typically modeled by neutron transport systems, where the trajectory of the particle
is composed of straight exponential paths between random changes of directions~\cite{dautray-lions-93-vol-6,Zoia2011}. An important
problem to design nuclear devices is the so-called shielding structure, aiming to protect humans from ionizing particles. It is in
particular crucial to compute the probability that a neutron exits the shielding structure $D$ before its absorption by the
medium~\cite{BachocBachouchLenotre2014}. This question is of course related to the quasi-stationary behavior of neutron tranport,
where absorption corresponds to the exit of a neutron from $D$.

We consider a piecewise-deterministic process of neutron transport with constant velocity. Let $D$ be an open connected bounded domain of
$\mathbb{R}^2$, let $S^2$ be the unit sphere of $\mathbb{R}^2$ and $\sigma(du)$ be the uniform probability measure on $S^2$. We
consider the Markov process $(X_t,V_t)_{t\geq 0}$ in $D\times S^2$ constructed as follows: $X_t=\int_0^t V_s\,ds$ and the velocity
$V_t\in S^2$ is a pure jump Markov process, with constant jump rate $\lambda>0$ and uniform jump probability distribution $\sigma$.
In other words, $V_t$ jumps to i.i.d. uniform values in $S^2$ at the jump times of a Poisson process. At the first time where
$X_t\not\in D$, the process immediately jumps to the cemetery point $\d$, meaning that the process is absorbed at the boundary of
$D$. An example of path of the process $(X,V)$ is shown in Fig.~\ref{fig:sample-path}. For all $x\in D$ and $u\in S^2$, we denote by
$\mathbb{P}_{x,u}$ (resp.\ $\mathbb{E}_{x,u}$) the distribution of $(X,V)$ conditionned on $(X_0,V_0)=(x,u)$ (resp.\ the expectation
with respect to $\mathbb{P}_{x,u}$).

\begin{figure}[h]
  \begin{center}
    \vspace{0.5cm}

    \includegraphics[height=6.5cm]{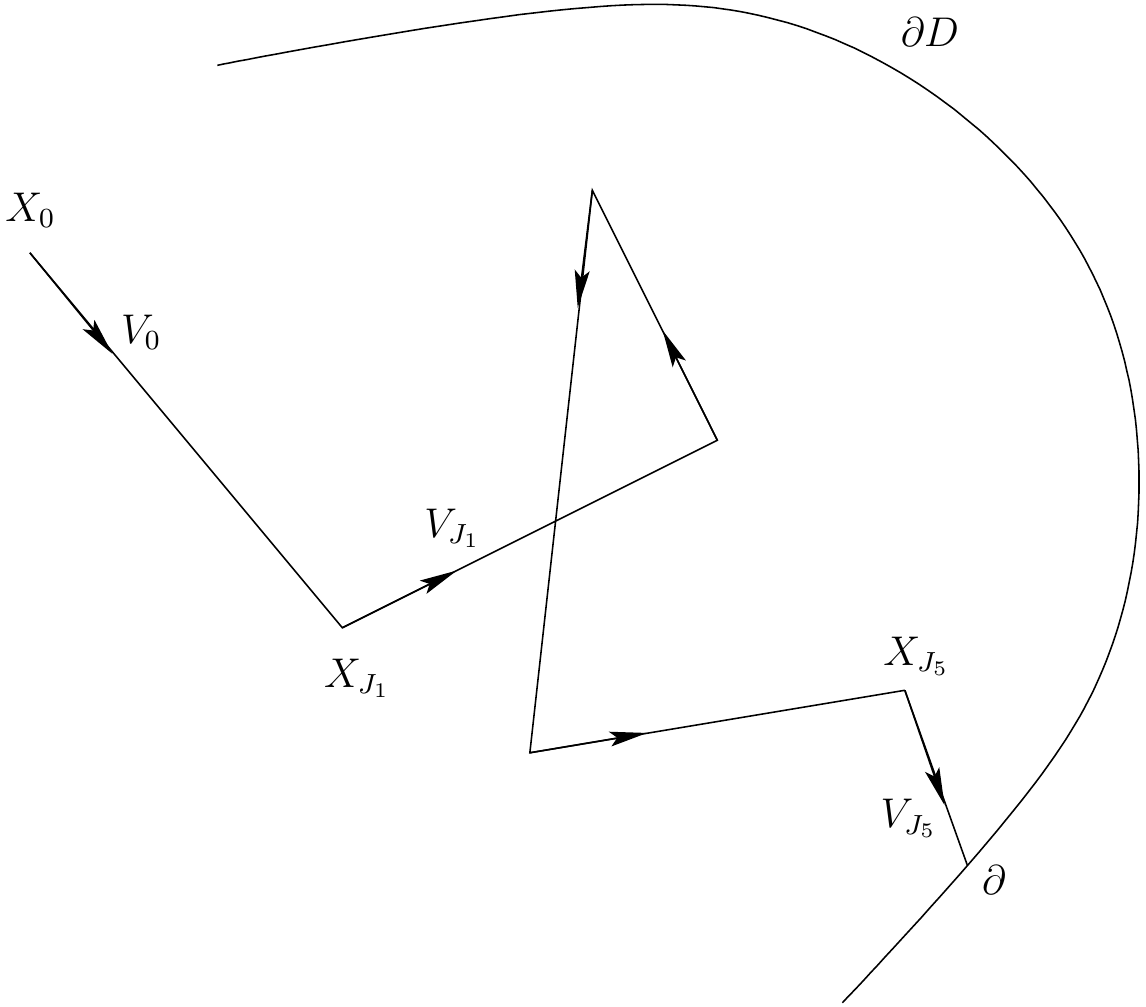}
  \end{center}
  \caption{A sample path of the neutron transport process $(X,V)$. The times $J_1<J_2<\ldots$ are the successive jump times of $V$.}
  \label{fig:sample-path}
\end{figure}

\begin{rem}
  The assumptions of contant velocity, uniform jump distribution, uniform jump rates and on the dimension of the process can be
  relaxed but we restrict here to the simplest case to illustrate how conditions~(A1) and~(A2) can be checked. In particular, it is
  easy to extend our results to variants of the process where, for instance, the jump measure for $V$ may depend on the state of the process, provided this measure is absolutely continuous w.r.t.\ $\sigma$ with density uniformly bounded from above and below.
\end{rem}

We denote by $\partial D$ the boundary of the domain $D$, $\text{diam}(D)$ its diameter, and for all $A\subset\mathbb{R}^2$ and
$x\in\mathbb{R}^2$, by $d(x,A)$ the distance of $x$ to the set $A$: $d(x,A)=\inf_{y\in A}|x-y|$. We also denote by $B(x,r)$ the open
ball of $\mathbb{R}^2$ centered at $x$ with radius $r$.  We assume that the domain $D$ is connected and smooth enough, in the
following sense.

\paragraph{Assumption~(B)}
We assume that there exists $\varepsilon>0$ such that
\begin{itemize}
\item[(B1)] $D_{\varepsilon}:=\{x\in D:d(x,\partial D)>\varepsilon\}$ is non-empty and connected;
\item[(B2)] there exists $0<s_\varepsilon<t_\varepsilon$ and $\underline{\sigma}>0$ such that, for all $x\in D\setminus
  D_\varepsilon$, there exists $K_x\subset S^2$ measurable such that $\sigma(K_x)\geq \underline{\sigma}$ and for all $u\in K_x$,
  $x+su\in D_\varepsilon$ for all $s\in[s_\varepsilon,t_\varepsilon]$ and $x+su\not\in \partial D$ for all $s\in[0,s_\varepsilon]$.
\end{itemize}

As illustrated by Fig.~\ref{fig:cone}, assumption~(B2) means that, for all $x \in D\setminus D_\varepsilon$, the set
$$
L_x:=\left\{y\in\mathbb{R}^2:|y-x|\in[s_\varepsilon,t_\varepsilon]\text{\ and\
}\frac{y-x}{|y-x|}\in K_x\right\}
$$
is included in $D_\varepsilon$ and has Lebesgue measure larger than $\frac{\underline{\sigma}}{2}(t_\varepsilon^2-s_\varepsilon^2)>0$.

\begin{figure}[h]
  \begin{center}
    \includegraphics[height=5.5cm]{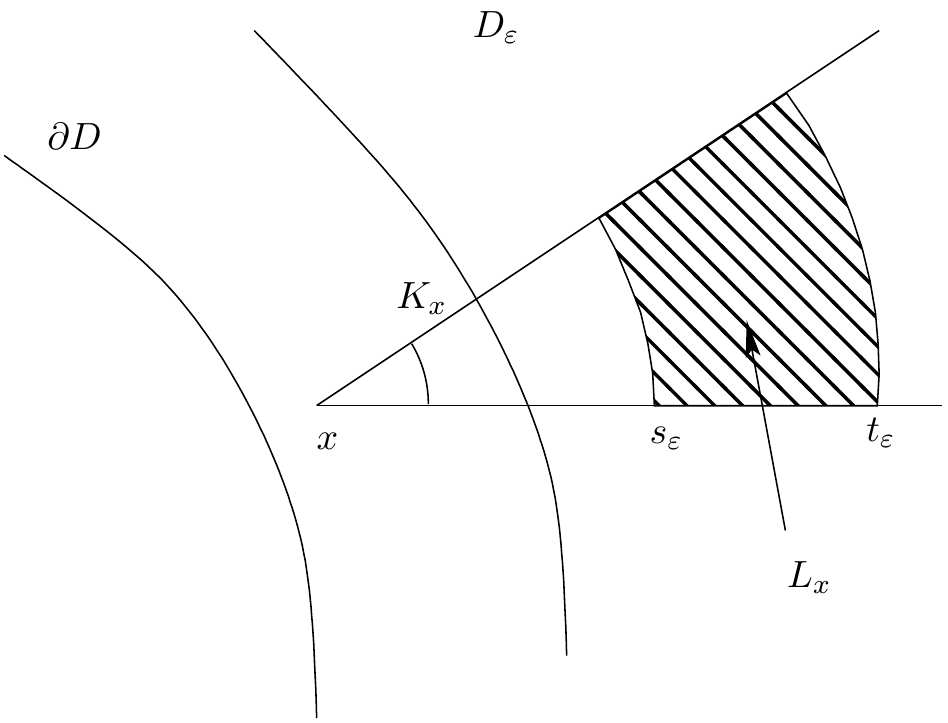}
  \end{center}
  \caption{The sets $K_x$ and $L_x$ of Assumption~(B2).}
  \label{fig:cone}
\end{figure}

These assumptions are true for example if $\partial D$ is a $C^2$ connected compact manifold, since then the so-called interior
sphere condition entails the existence of a cone $K_x$ satisfying~(B2) provided $\varepsilon$ is small enough compared to the maximum
curvature of the manifold.

\begin{thm}
  \label{thm:neutron}
  Assumption {\upshape (B)} implies {\upshape (i--vi)} in Theorem~\ref{thm:QSD_full}.
\end{thm}

\begin{proof}
In all the proof, we will make use of the following notation: for all $k\geq 1$, let $J_k$ be the $k$-th jump time of $V_t$ (the
absorption time is not considered as a jump, so $J_{k+1}=\infty$ if $J_k<\infty$ and $X_t$ hits $\partial D$ after $J_k$ and before
the $(k+1)$-th jump of $V_t$).

Let us first prove (A1). The following properties are easy consequences of the boundedness of $D$ and Assumption~(B).
\begin{lem}
  \label{prop:neutron}
  \begin{description}
  \item[\textmd{(i)}] There exists $n\geq 1$ and $x_1,\ldots,x_n\in D_{\varepsilon}$ such that $D_{\varepsilon}\subset\bigcup_{i=1}^n
    B(x_i,\varepsilon/16)$.
  \item[\textmd{(ii)}] For all $x,y\in D_{\varepsilon}$, there exists $m\leq n$ and $i_1,\ldots, i_m$ distinct in $\{1,\ldots,n\}$
    such that $x\in B(x_{i_1},\varepsilon/16)$, $y\in B(x_{i_m},\varepsilon/16)$ and for all $1\leq j\leq m-1$,
    $B(x_{i_j},\varepsilon/16)\cap B(x_{i_{j+1}},\varepsilon/16)\not=\emptyset$.
  \end{description}
\end{lem}




The next lemma is proved just after the current proof.
\begin{lem}
  \label{lem:neutron}
  For all $x\in D$, $u\in S^2$ and $t>0$ such that $d(x,\partial D)>t$, the law of $(X_t,V_t)$ under $\mathbb{P}_{x,u}$ satisfies
  $$
  \mathbb{P}_{x,u}(X_t\in dz,\,V_t\in dv)\geq \frac{\lambda^2 e^{-\lambda t}}{4\pi t}\, \frac{(t-|z-x|)^2}{t+|z-x|}\mathds{1}_{z\in
    B(x,t)}\Lambda(dz)\sigma(dv),
  $$
  where $\Lambda$ is Lebesgue's measure on $\mathbb{R}^2$.
\end{lem}

This lemma has the following immediate consequence. Fix $i\not=j$ in $\{1,\ldots,n\}$ such that $B(x_i,\varepsilon/16)\cap
B(x_j,\varepsilon/16)\not=\emptyset$. Then, for all $x\in B(x_i,\varepsilon/16)$ and $u\in S^2$,
$$
\mathbb{P}_{x,u}(X_{\varepsilon/2}\in dz,\,V_{\varepsilon/2}\in dv)\geq C_\varepsilon\mathds{1}_{B(x_j,\varepsilon/8)\cup
  B(x_i,\varepsilon/8)}(z)\Lambda(dz)\,\sigma(dv),
$$
for a constant $C_\varepsilon>0$ independent of $x$, $i$ and $j$.

Combining this result with Lemma~\ref{prop:neutron}, one easily deduces that for all $x\in D_\varepsilon$, $u\in S^2$ and $m\geq n$,
\begin{equation*}
  \mathbb{P}_{x,u}(X_{m\varepsilon/2}\in dz,\,V_{m\varepsilon/2}\in dv)\geq
  C_\varepsilon c^{m-1}_\varepsilon\mathds{1}_{D_\varepsilon}(z)\Lambda(dz)\,\sigma(dv),
\end{equation*}
where $c_\varepsilon=C_\varepsilon\Lambda(B(\varepsilon/16))=C_\varepsilon\pi\varepsilon^2/256$. Proceeding similarly, but with a
first time step of length in $[\varepsilon/2,\varepsilon)$, we can also
deduce from Lemma~\ref{lem:neutron} that, for all $t\geq n\varepsilon/2$,
\begin{equation}
  \label{eq:neutron-1}
  \mathbb{P}_{x,u}(X_{t}\in dz,\,V_{t}\in dv)\geq C'_\varepsilon c^{\lfloor
    2t/\varepsilon\rfloor-1}_\varepsilon\mathds{1}_{D_\varepsilon}(z)\Lambda(dz)\,\sigma(dv)
\end{equation}
for a contant $C'_\varepsilon>0$.

This entails~(A1) with $\nu$ the uniform probability measure on $D_\varepsilon\times S^2$ and any $t_0\geq n\varepsilon/2$, but only
for initial conditions in $D_\varepsilon\times S^2$.

Now, assume that $x\in D\setminus D_\varepsilon$ and $u\in S^2$. Let 
$$
s=\inf\{t\geq 0: x+tu\in D_\varepsilon\cup\partial D\}.
$$
If $x+su\in \d D_\varepsilon$, then $\mathbb{P}_{x,u}(X_{s}\in \d D_\varepsilon)\geq e^{-\lambda s}\geq
e^{-\lambda\,\text{diam}(D)}$, and thus, combining this with~\eqref{eq:neutron-1}, for all $t> n\varepsilon/2$,
\begin{multline*}
  \mathbb{P}_{x,u}(X_{s+t}\in dz,\,V_{s+t}\in dv\mid s+t<\tau_\d)\geq
  \mathbb{P}_{x,u}(X_{s+t}\in dz,\,V_{s+t}\in dv) \\ \geq
  e^{-\lambda\,\text{diam}(D)} C'_\varepsilon c^{\lfloor 2t/\varepsilon\rfloor-1}_\varepsilon\mathds{1}_{D_\varepsilon}(z)\Lambda(dz)\,\sigma(dv).  
\end{multline*}
If $x+su\in\partial D$,
\begin{align*}
  \mathbb{P}_{x,u}(X_{t_\varepsilon}\in D_\varepsilon,\,t_\varepsilon<\tau_\d) & \geq
  \mathbb{P}(J_1<s\wedge(t_\varepsilon-s_\varepsilon),\,V_{J_1}\in K_{x+uJ_1},\,J_2>t_\varepsilon) \\ & \geq
  \underline{\sigma}e^{-\lambda t_\varepsilon} \mathbb{P}(J_1<s\wedge(t_\varepsilon-s_\varepsilon)).
\end{align*}
Hence~\eqref{eq:neutron-1} entails, for all $t\geq n\varepsilon/2$ such that $t+t_\varepsilon\geq s$,
\begin{multline*}
  \mathbb{P}_{x,u}(X_{t_\varepsilon+t}\in dz,\,V_{t_\varepsilon+t}\in dv\mid t_\varepsilon+t<\tau_\d) \\ 
  \begin{aligned}
    & \geq \frac{\mathbb{P}_{x,u}(X_{t_\varepsilon}\in D_\varepsilon,\,t_\varepsilon<\tau_\d,\,X_{t_\varepsilon+t}\in
      dz,\,V_{t_\varepsilon+t}\in dv)}{\mathbb{P}_{x,u}(t+t_\varepsilon<\tau_\d)} 
    \\ & \geq \frac{\mathbb{P}(J_1<s\wedge(t_\varepsilon-s_\varepsilon))}{\mathbb{P}(J_1<s)}
    \underline{\sigma}e^{-\lambda t_\varepsilon} c^{\lfloor
      2t/\varepsilon\rfloor+1}_\varepsilon\mathds{1}_{D_\varepsilon}(z)\Lambda(dz)\,\sigma(dv).
  \end{aligned}
\end{multline*}
Since $t_\varepsilon\leq \text{diam}(D)$, we have for all $0<s\leq \text{diam}(D)$
$$
\frac{\mathbb{P}(J_1<s\wedge(t_\varepsilon-s_\varepsilon))}{\mathbb{P}(J_1<s)}\geq
\frac{1-e^{-\lambda(t_\varepsilon-s_\varepsilon)}}{1-e^{-\lambda\,\text{diam}(D)}}>0.
$$
Hence, we have proved (A1) with $\nu$ the uniform probability measure on $D_\varepsilon\times S^2$ and
$t_0=\frac{n\varepsilon}{2}+\text{diam}(D)$.

Now we come to the proof of (A2). This can be done in two steps: first, we prove that for all $x\in D$ and $u\in S^2$,
\begin{equation}
  \label{eq:neutron-3}
  \mathbb{P}_{x,u}(J_4<\infty,\,X_{J_4}\in dz)\leq C\mathds{1}_D(z)\Lambda(dz)  
\end{equation}
for some constant $C$ independent of $x$ and $u$; second
\begin{equation}
  \label{eq:neutron-2}
  \mathbb{P}_\nu(J_1<\infty,\,X_{J_1}\in dz)\geq c\mathds{1}_D(z)\Lambda(dz)
\end{equation}
for some constant $c>0$.

Since for all $k\geq 1$, conditionally on $\{J_k<\infty\}$, $V_{J_k}$ is uniformly distributed on $S^2$ and independent of $X_{J_k}$,
this is enough to conclude as follows: by~\eqref{eq:neutron-3} and the inequality $J_4\leq 4\,\text{diam}(D)$ a.s. on $\tau_\d>t$, for all $t\geq
4\,\text{diam}(D)$,
\begin{align*}
  \mathbb{P}_{x,u}(t<\tau_\d) & \leq\mathbb{E}_{x,u}[\mathbb{P}_{X_{J_4},V_{J_4}}(t-4\,\text{diam}(D)<\tau_\d)] \\ & \leq
  C\iint_D\int_{S^2}\mathbb{P}_{z,v}(t-4\,\text{diam}(D)<\tau_\d)\sigma(dv)\Lambda(dz).
\end{align*}
Similarly,~\eqref{eq:neutron-2} entails that, for all $t\geq \text{diam}(D)$,
$$
\mathbb{P}_\nu(t<\tau_\d)\geq c\iint_D\int_{S^2}\mathbb{P}_{z,v}(t<\tau_\d)\sigma(dv)\Lambda(dz),
$$
and thus, for all $t\geq 5\,\text{diam}(D)$,
$$
\mathbb{P}_{x,u}(t<\tau_\d)\leq\frac{C}{c}\mathbb{P}_\nu(t-4\,\text{diam}(D)<\tau_\d).
$$
Now, it follows from (A1) that $\mathbb{P}_\nu((X_{t_0},V_{t_0})\in\cdot)\geq c_1\mathbb{P}_\nu(t_0<\tau_\d)\nu(\cdot)$ and thus
\begin{align*}
  \mathbb{P}_\nu(t-4\,\text{diam}(D)+t_0<\tau_\d) & =\mathbb{E}_\nu[\mathbb{P}_{X_{t_0},V_{t_0}}(t-4\,\text{diam}(D)<\tau_\d)] \\ &
  \geq c_1\mathbb{P}_\nu(t-4\,\text{diam}(D)<\tau_\d).
\end{align*}
Iterating this inequality as needed completes the proof of (A2).

So it only remains to prove~\eqref{eq:neutron-3} and~\eqref{eq:neutron-2}. We start with~\eqref{eq:neutron-3}. We denote by
$(\hat{X}_t,\hat{V}_t)_{t\geq 0}$ the neutron transport process in $\hat{D}=\mathbb{R}^2$, coupled with $(X,V)$ such that
$\hat{X}_t=X_t$ and $\hat{V}_t=V_t$ for all $t<\tau_\d$. We denote by $\hat{J}_1<\hat{J}_2<\ldots$ the jumping times of $\hat{V}$. It
is clear that $\hat{J}_k=J_k$ for all $k\geq 1$ such that $J_k<\infty$.

Now, $\hat{X}_{\hat{J}_4}=Y_1+Y_2+Y_3+Y_4$, where the r.v. $Y_1,\ldots,Y_4$ are independent, $Y_1=x+u Z$, where $Z$ is an exponential
r.v.\ of parameter $\lambda$, and $Y_2,Y_3,Y_4$ are all distributed as $V Z$, where $V$ is uniform on $S^2$ and independent of $Z$.
Using the change of variable from polar to Cartesian coordinates, one checks that $VZ$ has density $g(z)=\frac{\lambda
  e^{-\lambda|z|}}{2\pi|z|}$ w.r.t. $\Lambda(dz)$, so that $g\in L^{3/2}(\mathbb{R}^2)$. Applying twice Young's inequality, one has
$g*g*g\in L^\infty(\mathbb{R}^2)$. Hence, for all $f\in \mathcal{B}(\mathbb{R}^2)$,
$$
\mathbb{E}_{x,u}[f(X_{J_4});\,J_4<\tau_\d]\leq\mathbb{E}[f(\hat{X}_{\hat{J}_4})]\leq \|g*g*g\|_{\infty}\iint_{\mathbb{R}^2}f(z)\Lambda(dz).
$$
Hence~\eqref{eq:neutron-3} is proved.

We finally prove~\eqref{eq:neutron-2}. For all $x,y\in\mathbb{R}^2$, we denote by $[x,y]$ the segment delimited by $x$ and $y$. For
all $f\in\mathcal{B}(\mathbb{R}^2)$,
\begin{align*}
  \mathbb{E}_\nu[f(X_{J_1});\,J_1<\infty] & =\iint_{D_\varepsilon}\frac{\Lambda(dx)}{\Lambda(D_\varepsilon)}\int_{S^2}\sigma(du)\int_0^\infty
  \mathds{1}_{[x,x+su]\subset D}\,\lambda e^{-\lambda s}f(x+su)\,ds \\
  & =\iint_{D_\varepsilon}\frac{\Lambda(dx)}{\Lambda(D_\varepsilon)}\iint_{D}\Lambda(dz)\mathds{1}_{[x,z]\subset
    D}\,\frac{\lambda e^{-\lambda |z-x|}}{2\pi|z-x|}f(z) \\
  & \geq \frac{\lambda e^{-\lambda\,\text{diam}(D)}}{2\pi\text{diam}(D) \Lambda(D_\varepsilon)}\iint_{D}\Lambda(dz) f(z)\iint_{D_\varepsilon}\Lambda(dx)\mathds{1}_{[x,z]\subset
    D}.
\end{align*}
Now, for all $z\in D\setminus D_\varepsilon$, using assumption~(B2),
$$
\iint_{D_\varepsilon}\mathds{1}_{[x,z]\subset
  D}\,\Lambda(dx)\geq\Lambda(L_z)\geq\frac{\underline{\sigma}}{2}(t_\varepsilon^2-s_\varepsilon^2),
$$
and for all $z\in D_\varepsilon$,
$$
\iint_{D_\varepsilon}\mathds{1}_{[x,z]\subset
  D}\,\Lambda(dx)\geq\Lambda(D_\varepsilon\cap B(z,\varepsilon)).
$$
Since the map $z\mapsto\Lambda(D_\varepsilon\cap B(z,\varepsilon))$ is continuous and positive on the compact set
$\overline{D_\varepsilon}$, we have proved~\eqref{eq:neutron-2}.
\end{proof}

\begin{proof}[Proof of Lemma~\ref{lem:neutron}]
  Using twice the relation: for all bounded measurable $f$, $k\geq 1$, $t\geq 0$, $x\in D$ and $u\in S^2$,
\begin{multline*}
  \mathbb{E}_{x,u}[f(X_t,V_t);J_k\leq t<J_{k+1}] \\ =\int_0^t ds\lambda e^{-\lambda
    s}\int_{S^2}\sigma(dv)\mathbb{E}_{x+su,v}[f(X_{t-s},V_{t-s});J_{k-1}\leq t-s< J_k],
\end{multline*}
we obtain
\begin{multline*}
  \mathbb{E}_{x,u}[f(X_t,V_t);J_2\leq t< J_3]=\lambda^2 e^{-\lambda t} \\ \int_{S^2}\sigma(dv) \int_{S^2}\sigma(dw)
  \int_0^t ds\int_0^{t-s} d\theta f(x+su+\theta v+(t-s-\theta)w,w).
\end{multline*}
For all $x,y,z\in \mathbb{R}^2$, we denote by $[x,y,z]$ the triangle of $\mathbb{R}^2$ delimited by $x$, $y$ and $z$. Using the
well-known fact that a point in $[x,y,z]$ with barycentric coordinates distributed uniformly on the simplex is distributed uniformly
on $[x,y,z]$, we deduce that
\begin{multline*}
  \mathbb{E}_{x,u}[f(X_t,V_t);J_2\leq t< J_3] \\ =\frac{\lambda^2t^2}{2} e^{-\lambda t} \int_{S^2}\sigma(dv)
  \int_{S^2}\sigma(dw)\iint_{[u,v,w]} f(x+tz,w)\,\frac{\Lambda(dz)}{\Lambda([u,v,w])}.
\end{multline*}

Now, for all $u,v,w\in S^2$,
$$
\Lambda([u,v,w])=\frac{1}{2}|u-w|\,|v-v'|\leq |u-w|,
$$
where $v'$ is the orthogonal projection of $v$ on the line $(u,w)$ (see fig.~\ref{fig:area-triangle}), and where we used the fact
that $|v-v'|\leq 2$.

\begin{figure}[h]
  \begin{center}
    \includegraphics[height=8cm]{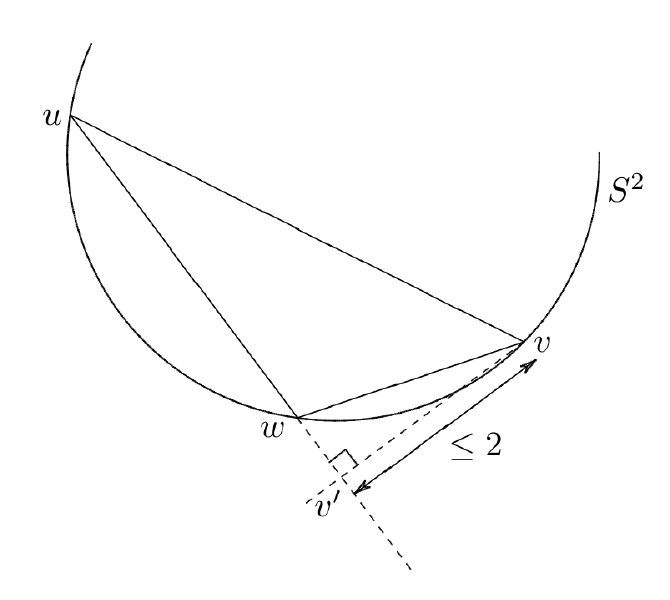}
  \end{center}
  \caption{The triangle $[u,v,w]$ and the point $v'$.}
  \label{fig:area-triangle}
\end{figure}

Moreover, for fixed $t\geq 0$, $x\in D$, $u,w\in S^2$,
\begin{multline*}
  \int_{S^2}\sigma(dv)\iint_{[u,v,w]}
  f(x+tz,w)\,\Lambda(dz) =\frac{1}{2\pi}\int_0^{2\pi}d\theta\iint_{[u,v_\theta,w]}f(x+tz,w)\Lambda(dz) \\
  \begin{aligned}
    & =\frac{1}{2\pi}\int_0^{2\pi}d\theta\iint_{B(1)}f(x+tz,w)\mathds{1}_{z\in[u,v_\theta,w]}\Lambda(dz) \\
    & =\iint_{B(1)} f(x+tz,w)\frac{\wideparen{u_zw_z}}{2\pi}\Lambda(dz), 
  \end{aligned}
\end{multline*}
where $v_\theta=(\cos\theta,\sin\theta)\in\mathbb{R}^2$, $B(r)$ is the ball centered at 0 of radius $r$ of $\mathbb{R}^2$, $u_z$
(resp.\ $w_z$) is the symmetric of $u$ (resp.\ $w$) with respect to $z$ in $S^2$ (see Fig.~\ref{fig:arc}) and $\wideparen{uv}$ is the
length of the arc between $u$ and $v$ in $S^2$.

\begin{figure}[h]
  \begin{center}
    \includegraphics[height=8cm]{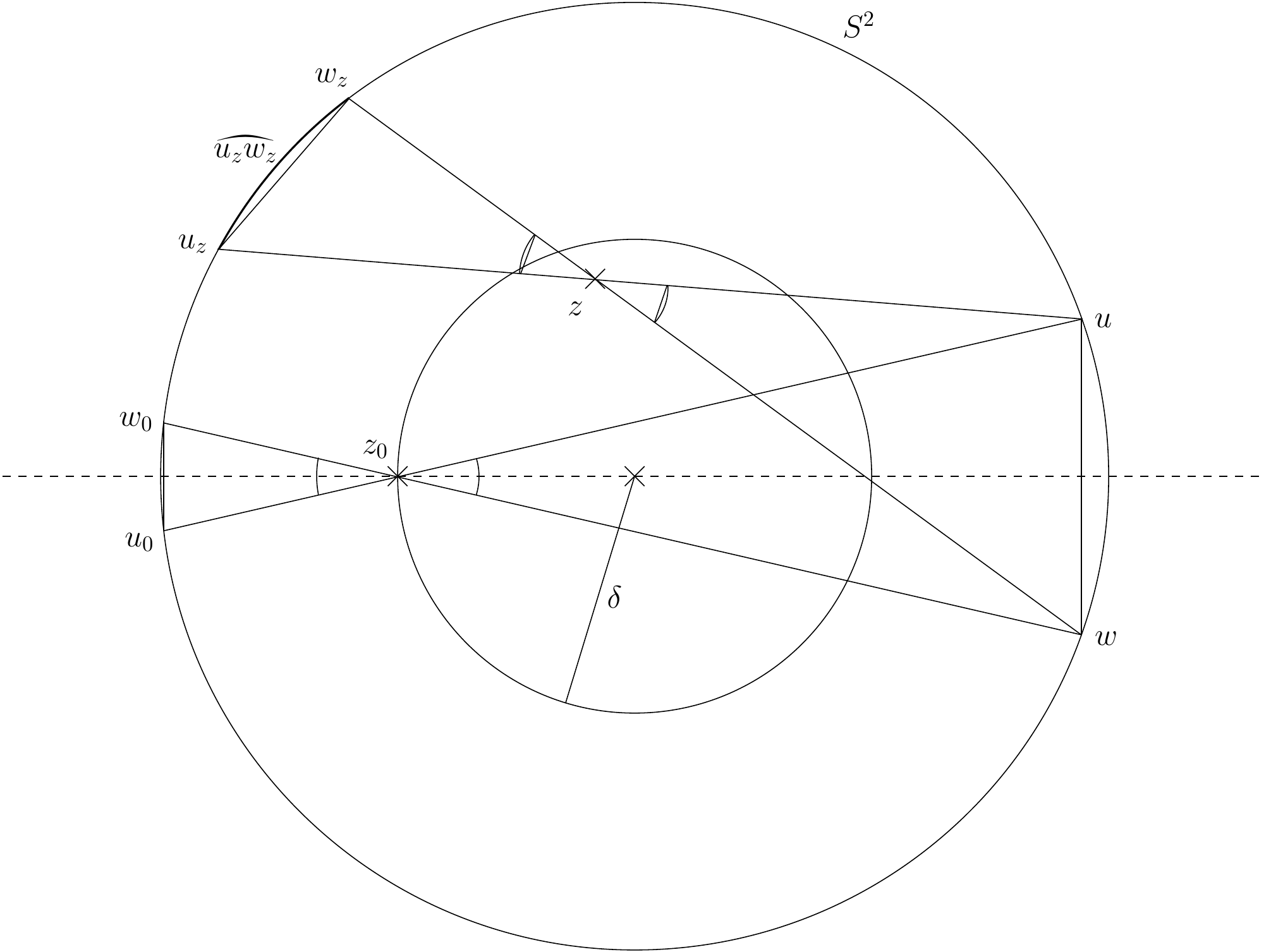} 
  \end{center}
  \caption{Definition of $u_z$, $w_z$, $\wideparen{u_zw_z}$, $z_0$, $u_0$ and $w_0$. The angle marked $\rrparenthesis$ is larger than
    the one marked $)$.}
  \label{fig:arc}
\end{figure}

Fix $0<\delta<1$, and let $z_0$ be the farthest point in $B(\delta)$ from $\{u,w\}$ (this point is unique except when the segment
$[u,w]$ between $u$ and $w$ is a diameter of $S^2$). We set $u_0=u_{z_0}$ and $w_0=w_{z_0}$ (see Fig.~\ref{fig:arc}). Note that Thales' theorem implies that
$$
|u_0-w_0|=\frac{1-\delta}{1+\delta}|u-w|.
$$
Then, for any $z\in B(\delta)$, we have $\widehat{uz_0w}\leq\widehat{uzw}$, where $\widehat{xyz}$ is the measure of the angle formed
by the segments $[y,x]$ and $[y,z]$. Since in addition $|z_0-u_0|=|z_0-w_0|\leq 1$ and $|z-u_z|\wedge|z-w_z|\geq 1-\delta$, Thales'
theorem yields
$$
\frac{|u_z-w_z|}{1-\delta}\geq |u_0-w_0|.
$$
%
%
%
Puting everything together, we deduce that
\begin{align*}
  \mathbb{E}_{x,u}[f(X_t,V_t); & J_2\leq t< J_3] \\
  & \geq\frac{\lambda^2t^2}{4\pi} e^{-\lambda t} \int_{S^2}\sigma(dw)\iint_{B(1)}f(x+tz,w)\frac{|u_z-w_z|}{|u-w|}\Lambda(dz) \\
  & \geq\frac{\lambda^2t^2}{4\pi} e^{-\lambda t} \int_{S^2}\sigma(dw)\iint_{B(1)}f(x+tz,w)\frac{(1-|z|)^2}{1+|z|}\Lambda(dz).
\end{align*}
This ends the proof of Lemma~\ref{lem:neutron}.
\end{proof}



\section{Quasi-stationary
  distribution: proofs of the results of Section~\ref{sec:results-QSD}}
\label{sec:QSD}

This section is devoted to the proofs of Theorem~\ref{thm:QSD_full} (Sections~\ref{sec:ii-implies-iv} and~\ref{sec:vi-implies-iii}), Corollary~\ref{cor:1} (Section~\ref{sec:proof-of-cor1}), Proposition~\ref{prop:2} (Section~\ref{sec:pf-prop:2}) and Corollary~\ref{cor:2} (Section~\ref{sec:proof-of-cor2}).
In Theorem~\ref{thm:QSD_full}, the implications (iii)$\Rightarrow$(i)$\Rightarrow$(ii) and (iv)$\Rightarrow$(v)$\Rightarrow$(vi) are obvious so we only need to prove (ii)$\Rightarrow$(iv) and (vi)$\Rightarrow$(iii).

%


\subsection{(ii) implies~(iv)}
\label{sec:ii-implies-iv}
Assume that $X$ satisfies Assumption~(A'). We shall prove the result assuming (A1') holds for $t_0=1$. The extension to any $t_0$ is immediate.

\medskip
\noindent
\textit{Step 1: control of the distribution at time 1, conditionally on non-absorption at a later time.}\\
Let us show that, for all $t\geq 1$ and for all $x_1,x_2\in E$, there exists a probability measure $\nu^t_{x_1,x_2}$ on $E$ such that, for all measurable set $A\subset E$,
\begin{align}
\label{eq:control-of-distribution}
\PP_{x_i}\left(X_1\in A\mid t<\tau_\partial\right)\geq c_1c_2\nu^t_{x_1,x_2}(A),\text{ for }i=1,2.
\end{align}
Fix $x_1,x_2\in E$, $i\in\{1,2\}$, $t\geq 1$ and a measurable subset $A\subset E$. Using the Markov property, we have
\begin{align*}
\PP_{x_i}\left(X_1\in A\text{ and }t<\tau_\partial\right)
&=\EE_{x_i}\left[\11_A(X_1)\PP_{X_1}\left(t-1<\tau_{\partial}\right)\right]\\
&=\EE_{x_i}\left[\11_A(X_1)\PP_{X_1}\left(t-1<\tau_{\partial}\right)\mid 1<\tau_{\partial}\right]\PP_{x_i}\left(1<\tau_{\partial}\right)\\
&\geq c_1\nu_{x_1,x_2}\left(\11_A(\cdot)\PP_{\cdot}\left(t-1<\tau_{\partial}\right)\right)\PP_{x_i}\left(1<\tau_{\partial}\right),
\end{align*}
by Assumption~(A1'). 
Dividing both sides by $\PP_{x_i}\left(t<\tau_{\partial}\right)$, we deduce that
\begin{align*}
\PP_{x_i}\left(X_1\in A\mid t<\tau_\partial\right)
&\geq 
c_1\nu_{x_1,x_2}\left(\11_A(\cdot)\PP_{\cdot}\left(t-1<\tau_{\partial}\right)\right)
\frac{\PP_{x_i}\left(1<\tau_{\partial}\right)}
{\PP_{x_i}\left(t<\tau_{\partial}\right)}.
\end{align*}
Using again the Markov property, we have
\begin{align*}
\PP_{x_i}\left(t<\tau_{\partial}\right)\leq \PP_{x_i}\left(1<\tau_{\partial}\right)\sup_{y\in E}\PP_y\left(t-1<\tau_{\partial}\right),
\end{align*}
so that
\begin{align*}
\PP_{x_i}\left(X_1\in A\mid t<\tau_\partial\right)
&\geq c_1
\frac{\nu_{x_1,x_2}\left(\11_A(\cdot)\PP_{\cdot}\left(t-1<\tau_{\partial}\right)\right)}
{\sup_{y\in E}\PP_y\left(t-1<\tau_{\partial}\right)}.
\end{align*}
Now Assumption~(A2') implies that the non-negative measure
\begin{align*}
B\mapsto \frac{\nu_{x_1,x_2}\left(\11_B(\cdot)\PP_{\cdot}\left(t-1<\tau_{\partial}\right)\right)}
{\sup_{y\in E}\PP_y\left(t-1<\tau_{\partial}\right)}
\end{align*}
has a total mass greater than $c_2$. Therefore~\eqref{eq:control-of-distribution} holds with
\begin{align*}
\nu^t_{x_1,x_2}(B)= \frac{\nu_{x_1,x_2}\left(\11_B(\cdot)\PP_{\cdot}\left(t-1<\tau_{\partial}\right)\right)}
{\PP_{\nu_{x_1,x_2}}\left(t-1<\tau_{\partial}\right)}
\end{align*}

\bigskip
\noindent
\textit{Step 2: exponential contraction for Dirac initial distributions}\\
 We now prove that, for all $x,y\in E$ and $T\geq 0$
\begin{align}
\label{eq:thm1-step2}
\left\|\PP_x\left(X_T\in \cdot\mid T<\tau_{\partial}\right)-\PP_y\left(X_T\in \cdot\mid T<\tau_{\partial}\right)\right\|_{TV}\leq 2(1-c_1c_2)^{\lfloor T \rfloor}.
\end{align}
Let us define, for all $0 \leq s \leq
 t\leq T$ the linear operator $R_{s,t}^T$ by
\begin{align*}
  R_{s,t}^T f(x)&= \EE_x(f(X_{t-s})\mid T-s< \tau_{\partial})\\
                &= \EE(f(X_{t})\mid X_s=x,\ T< \tau_{\partial}),
\end{align*}
by the Markov property. For any $T>0$, the
family $(R_{s,t}^T)_{0\leq s\leq t\leq T}$ is a Markov semi-group:
we
have, for all $0\leq u\leq s\leq t\leq T$ and all bounded measurable function $f$,
 \begin{equation*}
   R_{u,s}^T (R_{s,t}^T f)(x) = R_{u,t}^T f(x).
 \end{equation*}
This can be proved as Lemma~12.2.2 in~\cite{DelMoral2013} or by observing that a Markov process conditioned by an event in its tail $\sigma
$-field remains Markovian (but no longer time-homogeneous).


For any $x_1,x_2\in E$, we have by~\eqref{eq:control-of-distribution} that $\delta_{x_i} R_{s,s+1}^T-c_1 c_2\nu^{T-s}_{x_1,x_2}$ is a
positive measure whose mass is  $1-c_1 c_2$, for $i=1,2$. We deduce that
 \begin{multline*}
   \left\|\delta_{x_1} R_{s,s+1}^T-\delta_{x_2} R_{s,s+1}^T \right\|_{TV}\\
   \begin{aligned}
   &\leq
   \|\delta_{x_1} R_{s,s+1}^T - c_1 c_2 \nu^{T-s}_{x_1,x_2}\|_{TV}
   +\|\delta_{x_2} R_{s,s+1}^T - c_1 c_2 \nu^{T-s}_{x_1,x_2}\|_{TV}\\
   &\leq
   2(1-c_1 c_2).
   \end{aligned}
 \end{multline*}

Let $\mu_1,\mu_2$ be two
 mutually singular probability measures on $E$ and any $f\geq 0$, we have
 \begin{align*}
   \left\|\mu_1 R_{s,s+1}^T-\mu_2 R_{s,s+1}^T \right\|_{TV}
   &\leq \iint_{E^2} 
   \left\|\delta_{x} R_{s,s+1}^T-\delta_{y} R_{s,s+1}^T \right\|_{TV}
   d\mu_1\otimes d\mu_2(x,y)\\
   &\leq 2(1-c_1c_2)=(1-c_1c_2)\|\mu_1-\mu_2\|_{TV}.
   \end{align*}

 Now if $\mu_1$ and $\mu_2$ are any two different probability measures (not necessarily mutually singular), one can apply the previous result to the
 mutually singular probability measures
 $\frac{(\mu_1-\mu_2)_+}{(\mu_1-\mu_2)_+(E)}$ and
 $\frac{(\mu_1-\mu_2)_-}{(\mu_1-\mu_2)_-(E)}$. Then
 \begin{multline*}
  \left\|\frac{(\mu_1-\mu_2)_+}{(\mu_1-\mu_2)_+(E)} R_{s,s+1}^T-\frac{(\mu_1-\mu_2)_-}{(\mu_1-\mu_2)_-(E)} R_{s,s+1}^T \right\|_{TV}\\
   \leq (1-c_1 c_2) \left\|\frac{(\mu_1-\mu_2)_+}{(\mu_1-\mu_2)_+(E)}-\frac{(\mu_1-\mu_2)_-}{(\mu_1-\mu_2)_-(E)}\right\|_{TV}.
 \end{multline*}
Since $\mu_1(E)=\mu_2(E)=1$, we have
 $(\mu_1-\mu_2)_+(E)=(\mu_1-\mu_2)_-(E)$. So multiplying the last inequality by $(\mu_1-\mu_2)_+(E)$, we
 deduce that
 \begin{multline*}
   \|(\mu_1-\mu_2)_+ R_{s,s+1}^T-(\mu_1-\mu_2)_- R_{s,s+1}^T \|_{TV}\\
   \leq (1-c_1c_2) \|(\mu_1-\mu_2)_+-(\mu_1-\mu_2)_-\|_{TV}.
 \end{multline*}
 Since $(\mu_1-\mu_2)_+-(\mu_1-\mu_2)_-=\mu_1-\mu_2$, we obtain
 \begin{equation*}
 \label{eq:ergodicite1}
   \|\mu_1 R_{s,s+1}^T-\mu_2 R_{s,s+1}^T \|_{TV}\leq (1-c_1c_2) \|\mu_1-\mu_2\|_{TV}.
 \end{equation*}
 Using the semi-group property of $(R_{s,t}^T)_{s,t}$,
 we deduce that, for any $x,y\in E$,
 \begin{align*}
   \|\delta_x R_{0,T}^T - \delta_y R_{0,T}^T\|_{TV}
   &=\|\delta_x R^T_{0,T-1} R_{T-1,T}^T - \delta_y R_{0,T-1}^T R_{T-1,T}^T\|_{TV}\\
           &\leq \left(1-c_1c_2\right)\|\delta_x R_{0,T-1}^T - \delta_y R_{0,T-1}^T\|_{TV}\\
           &\leq\ \ldots\ \leq  2 \left(1-c_1c_2\right)^{\lfloor T\rfloor}.
 \end{align*}
By definition of $R_{0,T}^T$, this inequality immediately leads to~\eqref{eq:thm1-step2}.

\bigskip
\noindent
\textit{Step 3: exponential contraction for general initial distributions}\\
We prove now that inequality~\eqref{eq:thm1-step2} extends to any pair of
 initial probability measures $\mu_1,\mu_2$ on $E$, that is,  for all $T\geq 0$,
\begin{align}
\label{eq:thm1-step3}
\left\|\PP_{\mu_1}\left(X_T\in \cdot\mid T<\tau_{\partial}\right)-\PP_{\mu_2}\left(X_T\in \cdot\mid T<\tau_{\partial}\right)\right\|_{TV}\leq 2(1-c_1c_2)^{\lfloor T \rfloor}.
\end{align}
Let $\mu_1$ be a probability measure on
 $E$ and $x\in E$. We have
\begin{align*}
  &\|\PP_{\mu_1}(X_T\in\cdot\mid T<\tau_\partial)-\PP_{x}(X_T\in\cdot\mid T<\tau_\partial)\|_{TV}\\
  &=\frac{1}{\PP_{\mu_1}(T<\tau_\partial)}\|\PP_{\mu_1}(X_T\in\cdot)-\PP_{\mu_1}(T<\tau_\partial)\PP_{x}(X_T\in\cdot\mid T<\tau_\partial)\|_{TV}\\
  &\leq
  \frac{1}{\PP_{\mu_1}(T<\tau_\partial)} \int_{y\in E}
  \|\PP_{y}(X_T\in\cdot)-\PP_{y}(T<\tau_\partial)\PP_x(X_T\in\cdot\mid T<\tau_\partial)\|_{TV}d\mu_1(y)\\
  &\leq   \frac{1}{\PP_{\mu_1}(T<\tau_\partial)} \int_{y\in E}
  \PP_{y}(T<\tau_\partial)\|\PP_{y}(X_T\in\cdot\mid T<\tau_\partial)-\PP_x(X_T\in\cdot\mid T<\tau_\partial)\|_{TV}d\mu_1(y)\\
  &\leq  \frac{1}{\PP_{\mu_1}(T<\tau_\partial)} \int_{y\in E}
  \PP_{y}(T<\tau_\partial) 2 (1-c_1c_2)^{\lfloor T\rfloor}d\mu_1(y)\\
  &\leq 2 (1-c_1c_2)^{\lfloor T\rfloor}.
\end{align*}
The same computation, replacing $\delta_x$ by any probability measure,
leads to~\eqref{eq:thm1-step3}.

\bigskip
\noindent
\textit{Step 4: existence and uniqueness of a quasi-stationary distribution for $X$.}\\
Let us first prove the uniqueness of the quasi-stationary distribution. If $\alpha_1$ and
$\alpha_2$ are two quasi-stationary distributions, then we have $\PP_{\alpha_i}(X_t\in\cdot|t<
\tau_{\partial})=\alpha_i$ for $i=1,2$ and any $t\geq 0$. Thus, we deduce from
inequality~\eqref{eq:thm1-step3} that
\begin{equation*}
  \|\alpha_1-\alpha_2\|_{TV}\leq 2 (1-c_1c_2)^{\lfloor t\rfloor},\ \forall t\geq 0,
\end{equation*}
which yields $\alpha_1=\alpha_2$.

 Let us now prove the existence of a QSD. By
\cite[Proposition~1]{MV12}, this is equivalent to prove the
existence of a quasi-limiting distribution for $X$.
So we only need to prove that
$\PP_x(X_t\in\cdot|t<\tau_{\partial})$ converges when $t$ goes to infinity, for some $x\in E$.
 We have, for all $s,t\geq 0$ and $x\in E$,
 \begin{align}
 \PP_x\left(X_{t+s}\in\cdot\mid t+s<\tau_{\partial}\right)
 &=\frac{\delta_x P_{t+s}}{\delta_x P_{t+s}\11_E}
 =\frac{\delta_x P_tP_s}{\delta_x P_t P_s\11_E}=\frac{\delta_x R_{0,s}^s P_t}{\delta_x R_{0,s}^sP_t\11_E}\notag\\
 &=\PP_{\delta_x R_{0,s}^s}\left(X_t\in\cdot\mid t<\tau_{\partial}\right),\label{eq:semigroup2}
 \end{align}
 where we use the identity $R_{0,s}^sf(x)=\frac{P_s f(x)}{P_s\11_E(x)}$ for the third equality.
Hence
\begin{align*}
  \|\PP_x(X_t\in\cdot|t<\tau_{\partial})&-\PP_x(X_{t+s}\in\cdot|t+s<\tau_{\partial})\|_{TV}\\
  &=\|\PP_x(X_t\in\cdot|t<\tau_{\partial})-\PP_{\delta_x R_{0,s}^{s}}(X_{t}\in\cdot|t<\tau_{\partial})\|_{TV}\\
  &\leq 2 \left(1-c_1c_2\right)^{\lfloor t\rfloor}
  \xrightarrow[s,t\rightarrow+\infty]{} 0.
\end{align*}
In particular the sequence $(\PP_x(X_t\in\cdot\mid t<\tau_{\partial}))_{t\geq 0}$ is a Cauchy
sequence for the total variation norm. The space of probability
measures on $E$ equipped with the total variation norm is complete,
so $\PP_x(X_t\in\cdot\mid t<\tau_{\partial})$ converges when $t$ goes to
infinity to some probability measure $\alpha$ on $E$. 

Finally Equation~\eqref{eq:expo-cv-explicit} follows from~\eqref{eq:thm1-step3} with $\mu_1=\mu$ and $\mu_2=\alpha$. Therefore we have proved (iv) and the last statement of Theorem~\ref{thm:QSD_full} concerning existence and uniqueness of a quasi-stationary distribution and the explicit expression for $C$ and $\gamma$.



\subsection{(vi) implies (iii)}
\label{sec:vi-implies-iii}
Assume that \eqref{eq:epsilon-t} holds with some probability measure $\alpha$ on $E$. Let us define
\begin{align*}
\varepsilon(t)=\sup_{x\in E}\left\|\P_x(X_t\in\cdot\mid t<\tau_\d)-\alpha\right\|_{TV}.
\end{align*}

\bigskip
\noindent
\textit{Step 1: $\varepsilon(\cdot)$ is non-increasing and $\alpha$ is a quasi-stationary distribution}\\
For all $s,t\geq 0$, $x\in E$ and $A\in{\cal E}$,
\begin{align*}
|\P_x(X_{t+s}\in A &\mid t+s<\tau_\d)-\alpha(A)\\
&=\left|\frac{\E_x\left\{\11_{t<\tau_\d}\P_{X_t}(s<\tau_\d)\left[\P_{X_t}(X_s\in A\mid s<\tau_\d)-\alpha(A)\right]\right\}}{\P_x(t+s<\tau_\d)}\right|\\
&\leq \frac{\E_x\left\{\11_{t<\tau_\d}\P_{X_t}(s<\tau_\d)\left|\P_{X_t}(X_s\in A\mid s<\tau_\d)-\alpha(A)\right|\right\}}{\P_x(t+s<\tau_\d)}\\
&\leq \varepsilon(s).
\end{align*}
Taking the supremum over $x$ and $A$, we deduce that the function $\varepsilon(\cdot)$ is non-increasing. By~\eqref{eq:epsilon-t}, this implies that $\varepsilon(t)$ goes to $0$ when $t\rightarrow+\infty$. By Step~4 above, $\alpha$ is a quasi-stationary distribution and there exists $\lambda_0>0$ such $\PP_{\alpha}(t<\tau_{\partial})=e^{-\lambda_0 t}$.

\bigskip
\noindent
\textit{Step 2: proof of~(A$2''$) for $\mu=\alpha$}\\
We define, for all $s\geq 0$,
\begin{align*}
A(s)=\frac{\sup_{x\in E}\PP_x(s<\tau_{\partial})}{\PP_\alpha(s<\tau_{\partial}).}=e^{\lambda_0 s}\sup_{x\in E}\PP_x(s<\tau_{\partial}).
\end{align*}
Our goal is to prove that $A$ is bounded.
The Markov property implies, for $s\leq t$,
\begin{align*}
\PP_x(t<\tau_{\partial})=\PP_x(s<\tau_{\partial})\,\EE\left(\PP_{X_s}(t-s<\tau_{\partial})\mid s<\tau_{\partial}\right).
\end{align*}
By~\eqref{eq:epsilon-t}, the total variation distance between $\alpha$ and ${\cal L}_x(X_s\mid s<\tau_{\d})$ is smaller than $\varepsilon(s)$, so
\begin{align*}
\PP_x(t<\tau_{\partial})\leq \PP_x(s<\tau_{\partial})\,\left(\PP_{\alpha}(t-s<\tau_{\partial})+\varepsilon(s)\sup_{y\in E}\PP_y(t-s<\tau_\d)\right).
\end{align*}
For $s\leq t$, we thus have
\begin{align}
\label{eq:psi}
A(t)\leq A(s)\left(1+\varepsilon(s)A(t-s)\right).
\end{align}
The next lemma proves that
\begin{align}
\label{eq:A2-for-alpha}
\P_{\alpha}(t<\tau_\d)=e^{-\lambda_0 t}\geq c_2(\alpha)\sup_{x\in E}\P_x(t<\tau_\d)
\end{align}
for the constant $c_2(\alpha)=1/\sup_{s>0} A(s)$ and concludes Step~2.
\begin{lem}
\label{le:1}
A function $A:\R_+\mapsto\R_+$ satisfying \eqref{eq:psi} for all $s\leq t$ is bounded.
\end{lem}

\begin{proof}
We introduce the non-decreasing function
$
\psi(t)=\sup_{0\leq s\leq t} A_s.
$
It follows from~\eqref{eq:psi} that, for all $s\leq u\leq t$,
\begin{align*}
A(u)\leq \psi(s)\left(1+\varepsilon(s)\psi(t-s)\right).
\end{align*}
Since this inequality holds also for $u\leq s$, we obtain for all $s\leq t$,
\begin{align}
\psi(t)\leq \psi(s)\left(1+\varepsilon(s)\psi(t-s)\right).
\end{align}
By induction, for all $N\geq 1$ and $s\geq 0$,
\begin{align}
\psi(Ns)&\leq \psi(s)\prod_{k=1}^{N-1}\left(1+\varepsilon(ks)\psi(s)\right)\notag\\
&\leq \psi(s)\exp\left(\psi(s)\sum_{k=1}^{N-1} \varepsilon(ks)\right)\notag\\
&\leq \psi(s)\exp\left(\psi(s)\sum_{k=1}^{\infty} \varepsilon(ks)\right),\label{eq:psi-s-lemma}
\end{align}
where $\sum_{k=1}^{\infty} \varepsilon(ks)<\infty$ by~\eqref{eq:epsilon-t}.

Since $\psi$ is non-decreasing, it is bounded.

\end{proof}

\begin{rem}
  Note that, under Assumption~(v), $\varepsilon(t)\leq Ce^{-\gamma t}$. Using the fact that $\psi(s)\leq e^{\lambda_0 s}$, we deduce
  from~\eqref{eq:psi-s-lemma} that, for all $s>0$,
  \begin{align*}
    A(Ns)\leq \exp\left(
      \lambda_0s+\frac{Ce^{(\lambda_0-\gamma)s}}{1-e^{-\gamma s}}
    \right).
  \end{align*}
  This justifies~\eqref{eq:ineg-c2-de-alpha} in Remark~\ref{rem:explicit-constant-c_2}.
\end{rem}

\bigskip
\noindent
\textit{Step 3: proof of (A$2''$)}\\
Applying Step~3 of Section~\ref{sec:ii-implies-iv} with $\mu_1=\rho$ and $\delta_x=\alpha$, we easily obtain
\begin{align*}
\varepsilon(t)=\sup_{\rho\in{\cal M}_1(E)}\left\|\P_{\rho}(X_t\in\cdot\mid t<\tau_\d)-\alpha\right\|_{TV}.
\end{align*}
Let $\mu$ be a probability measure on $E$. For any $s\geq 0$, let us define $\mu_s(\cdot)=\P_{\mu}(X_s\in\cdot\mid s<\tau_\d)$.  We have $\|\mu_s-\alpha\|_{TV}\leq \varepsilon(s)$ and so
\begin{align*}
\P_{\mu_s}(t-s<\tau_\d)&\geq \P_{\alpha}(t-s<\tau_\d)-\varepsilon(s)\sup_{x\in E}\P_{x}(t-s<\tau_\d)\\
&\geq e^{-\lambda_0(t-s)}-\frac{\varepsilon(s)}{c_2(\alpha)}e^{-\lambda_0(t-s)}
\end{align*}
by~\eqref{eq:A2-for-alpha}.
Since $\varepsilon(s)$ decreases to $0$, there exists $s_0$ such that $\varepsilon(s_0)/c_2(\alpha)=1/2$. Using the Markov property, we deduce that, for any $t\geq s_0$,
\begin{align*}
\P_{\mu}(t<\tau_\d)&=\P_{\mu_{s_0}}(t-s_0<\tau_\d)\P_{\mu}(s_0<\tau_\d)\\
&\geq \frac{e^{-\lambda_0 (t-s_0)}}{2} \P_{\mu}(s_0<\tau_\d).
\end{align*}
Therefore, by~\eqref{eq:A2-for-alpha}, we have proved that
\begin{align*}
\P_{\mu}(t<\tau_\d)\geq c_2(\mu)\sup_{x\in E}\P_x(t<\tau_\d)
\end{align*}
for 
\begin{align*}
c_2(\mu)=\frac{1}{2}c_2(\alpha)e^{\lambda_0 s_0}\P_\mu(s_0<\tau_\d)>0.
\end{align*}

\bigskip
\noindent
\textit{Step 4: construction of the measure $\nu$ in (A1)}\\
\noindent We define the measure $\nu$ as the infimum of the family of measures $(\delta_x R_{0,2t}^{2t})_{x\in E}$, as defined in the next lemma, for a fixed $t$ such that $c_2(\alpha)\geq 2\varepsilon(t)$. To prove $(A1)$  for $t_0=2t$, we only need to check that $\nu$ is a positive measure.
\begin{lem}
\label{le:3}
Let $(\mu_x)_{x\in F}$ be a family of positive measures  on $E$ indexed by an arbitrary set $F$. For all $A\in {\cal E}$, we define
\begin{align*}
\mu(A)=\inf\left\{\sum_{i=1}^n \mu_{x_i}(B_i) \mid n\geq 1,\,x_1,\ldots,x_n\in F,\, B_1,\ldots,B_n\in{\cal E}\text{ partition of }A\right\}.
\end{align*}
Then $\mu$ is the largest non-negative measure on $E$ such that $\mu\leq \mu_x$ for all $x\in F$ and is called the infimum measure of $(\mu_x)_{x\in F}$.
\end{lem}

\begin{proof}[Proof of Lemma~\ref{le:3}]
Clearly $\mu(A)\geq 0$ for all $A\in{\cal E}$ and $\mu(\emptyset)=0$.
Let us prove the $\sigma$-additivity. Consider disjoints measurable sets $A_1,A_2,\ldots$ and define $A=\cup_{k} A_k$. Let $B_1,\ldots,B_n$ be a measurable partition of $A$. Then
\begin{align*}
\sum_{i=1}^n \mu_{x_i}(B_i)=\sum_{k=1}^\infty \sum_{i=1}^n \mu_{x_i}(B_i\cap A_k)\geq \sum_{k=1}^\infty \mu(A_k).
\end{align*}
Hence $\mu(A)\geq \sum_{k=1}^\infty \mu(A_k)$.

Fix $K\geq 1$ and $\epsilon>0$. For all $k\in\{1,\ldots,K\}$, let $(B_i^k)_{i\in \{1,\ldots,n_k\}}$ be a partition of $A_i$ and $(x^k_i)_{i\in \{1,\ldots,n_k\}}$ be such that
\begin{align*}
\mu(A_k)\geq \sum_{i=1}^{n_k} \mu_{x_i^k}(B_i^k)-\frac{\epsilon}{2^k}.
\end{align*}
Then
\begin{align*}
\sum_{k=1}^K\mu(A_k)\geq \sum_{k=1}^K \left(\sum_{i=1}^{n_k} \mu_{x_i^k}(B_i^k)-\frac{\epsilon}{2^k}\right)\geq \mu\left(\bigcup_{k=1}^K A_k\right)-\epsilon.
\end{align*}
Since, for any $x_0\in F$, 
\begin{align*}
\mu(A)\leq \mu\left(\bigcup_{k=1}^K A_k\right)+\mu_{x_0}\left(A\setminus \bigcup_{k=1}^K A_k\right),
\end{align*}
choosing $K$ large enough, $\mu(A)\leq \mu\left(\cup_{k=1}^K A_k\right)+\epsilon$.
Combining this with the previous inequality, we obtain
\begin{align*}
\mu(A)\geq \sum_{k=1}^\infty \mu(A_k)\geq \mu(A)-2\epsilon.
\end{align*}
This concludes the proof that $\mu$ is a non-negative measure.

Let us now prove that $\mu$ is the largest non-negative measure on $E$ such that $\mu\leq \mu_x$ for all $x\in F$. Let $\hat{\mu}$ be another measure such that $\hat{\mu}\leq \mu_x$. Then, for all $A\in{\cal E}$ and $B_1,\ldots,B_n$ a measurable partition of $A$ and $x_1,\ldots,x_n\in F$,
\begin{align*}
\hat{\mu}(A)=\sum_{i=1}^n\hat{\mu}(B_i)\leq \sum_{i=1}^n \mu_{x_i}(B_i).
\end{align*}
Taking the infimum over $(B_i)$ and $(x_i)$ implies that $\hat{\mu}\leq \mu$.
\end{proof}

\noindent Let us now prove that $\nu$ is a positive measure. By~\eqref{eq:semigroup2}, for any $x\in E$, $t\geq 0$ and $A\subset E$ measurable,
\begin{align*}
\delta_x R^{2t}_{0,2t}(A)&=\P_{\delta_x R^t_{0,t}}\left(X_t\in A|t<\tau_{\d}\right)\\
&=\frac{\int_E \P_y(X_t\in A)\delta_x R^t_{0,t}(dy)}{\int_E \P_y(t<\tau_\d)\delta_x R^t_{0,t}(dy)}.
\end{align*}
By Step~2, we have
\begin{align*}
\delta_x R^{2t}_{0,2t}(A)&\geq 
c_2(\alpha)e^{\lambda_0 t}\int_E \P_y(X_t\in A)\delta_x R^t_{0,t}(dy).
\end{align*}
We set $\nu_{t,x}^+=(\alpha-\delta_x R_{0,t}^t)_+$.
Using the inequality (between measures) $\delta_x R_{0,t}^t\geq \alpha-\nu_{t,x}^+$,
\begin{align*}
\delta_x R^{2t}_{0,2t}(A)&\geq c_2(\alpha)\alpha(A)-
c_2(\alpha)e^{\lambda_0 t}\int_E \P_y(X_t\in A\mid t<\tau_\d)\P_y(t<\tau_\d)\nu_{t,x}^+(dy).
\end{align*}
Since $\nu_{t,x}^+$ is a positive measure, Step~2 implies again
\begin{align*}
\delta_x R^{2t}_{0,2t}(A)&\geq c_2(\alpha)\alpha(A)-
\int_E \P_y(X_t\in A\mid t<\tau_\d)\nu_{t,x}^+(dy)\\
&= c_2(\alpha)\alpha(A)-
\int_E \delta_y R^t_{0,t}(A)\nu_{t,x}^+(dy)\\
&\geq \left(c_2(\alpha)-\nu_{t,x}^+(E)\right)\alpha(A)-\int_E \left(\delta_y R^t_{0,t}(A)-\alpha(A)\right)\nu_{t,x}^+(dy)\\
&\geq \left(c_2(\alpha)-\varepsilon(t)\right)\alpha(A)-\int_E \left(\delta_y R^t_{0,t}(A)-\alpha(A)\right)_+\nu_{t,x}^+(dy),
\end{align*}
where the inequality $\nu_{t,x}^+(E)\leq \varepsilon(t)$ follows from~\eqref{eq:epsilon-t}. Moreover $\nu_{t,x}^+=(\alpha-\delta_x R_{0,t}^t)_+\leq \alpha$, therefore
\begin{align*}
\delta_x R^{2t}_{0,2t}(A)&\geq \left(c_2(\alpha)-\varepsilon(t)\right)\alpha(A)-\int_E \left(\delta_y R^t_{0,t}(A)-\alpha(A)\right)_+\alpha(dy).
\end{align*}
Hence, for all $B_1,\ldots,B_n$ a measurable partition of $E$ and all $x_1,\ldots,x_n\in E$, 
\begin{align*}
\sum_{i=1}^n \delta_{x_i} R_{0,2t}^2t(B_i)\geq 
\left(c_2(\alpha)-\varepsilon(t)\right)\alpha(E)-\int_E \sum_{i=1}^n\left(\delta_y R^t_{0,t}(B_i)-\alpha(B_i)\right)_+\alpha(dy).
\end{align*}
Now 
\begin{align*}
\sum_{i=1}^n\left(\delta_y R^t_{0,t}(B_i)-\alpha(B_i)\right)_+
\leq \left\|\delta_y R^t_{0,t}(B_i)-\alpha(B_i)\right\|_{TV} \leq \varepsilon(t).
\end{align*}
Therefore $\nu(E)\geq c_2(\alpha)-2\varepsilon(t)>0$.

This concludes the proof of Theorem~\ref{thm:QSD_full}.
%
%

\subsection{Proof of Corollary~\ref{cor:1}}
\label{sec:proof-of-cor1}

Assume that the hypotheses of Theorem~\ref{thm:QSD_full} are satisfied and let $\mu_1,\mu_2$ be two probability measures and $t\geq 0$. Without loss of generality, we assume that $\P_{\mu_1}(t<\tau_\d)\geq \P_{\mu_2}(t<\tau_\d)$ and prove that
\begin{align*}
\left\|\P_{\mu_1}\left(X_t\in\cdot\mid t<\tau_\d\right)-\P_{\mu_2}\left(X_t\in\cdot\mid t<\tau_\d\right)\right\|\leq \frac{(1-c_1c_2)^{\lfloor t/t_0\rfloor}}{c_2(\mu_1)}\|\mu_1-\mu_2\|_{TV}.
\end{align*}
Using the relation
\begin{align*}
\mu_1 P_t=(\mu_1-(\mu_1-\mu_2)_+)P_t+(\mu_1-\mu_2)_+P_t,
\end{align*}
the similar one for $\mu_2 P_t$ and $\mu_1-(\mu_1-\mu_2)_+=\mu_2-(\mu_2-\mu_1)_+$, we can write
\begin{align}
\label{eq:mu1-mu2}
\frac{\mu_1 P_t}{\mu_1 P_t\11_E}-\frac{\mu_2 P_t}{\mu_2 P_t\11_E}=\alpha_1 P_t-\alpha_2 P_t,
\end{align}
where $\alpha_1$ and $\alpha_2$ are the positive measures defined by
\begin{align*}
\alpha_1=\frac{(\mu_1-\mu_2)_+}{\mu_1 P_t\11_E}
\end{align*}
and
\begin{align*}
\alpha_2=\frac{(\mu_2-\mu_1)_+}{\mu_2 P_t\11_E}+\left(\frac{1}{\mu_2 P_t\11_E}-\frac{1}{\mu_1 P_t\11_E}\right)\times\left(\mu_1-(\mu_1-\mu_2)_+\right).
\end{align*}
We immediately deduce from\eqref{eq:mu1-mu2}, that $\alpha_1P_t\11_E=\alpha_2P_t\11_E$, so that
\begin{align*}
\|\alpha_1 P_t-\alpha_2 P_t\|_{TV}&=\alpha_1 P_t\11_E\left\|\frac{\alpha_1}{\alpha_1P_t\11_E}P_t-\frac{\alpha_2}{\alpha_2P_t\11_E}P_t\right\|_{TV}\\
& \leq 2(1-c_1c_2)^{\lfloor t/t_0\rfloor}\alpha_1P_t\11_E.
\end{align*}
Since
\begin{align*}
\alpha_1 P_t\11_E&=\frac{(\mu_1-\mu_2)_+P_t\11_E}{\mu_1P_t\11_E}\\
&\leq (\mu_1-\mu_2)_+(E)\frac{\sup_{\rho\in{\cal M}_1(E)}\rho P_t\11_E}{\mu_1P_t\11_E}\leq \frac{\|\mu_1-\mu_2\|_{TV}}{2c_2(\mu_1)}
\end{align*}
by definition of $c_2(\mu_1)$, the proof of Corollary~\ref{cor:1} is complete.

\subsection{Proof of Proposition~\ref{prop:2}}
\label{sec:pf-prop:2}
\noindent \textit{Step 1: Existence of $\eta$.}\\
For all $x\in E$ and $t\geq 0$, we set
\begin{align*}
\eta_t(x)=\frac{\P_x(t<\tau_\d)}{\P_\alpha(t<\tau_\d)}=e^{\lambda_0 t}\P_x(t<\tau_\d).
\end{align*}
By the Markov property
\begin{align*}
\eta_{t+s}(x)&=e^{\lambda_0 (t+s)}\E_x\left(\11_{t<\tau_\d}\P_{X_t}(s<\tau_\d)\right)\\
&=\eta_t(x)\E_x\left(\eta_s(X_t)\mid t<\tau_\d\right).
\end{align*}
By (A$2''$), $\int_E\eta_s(y)\rho(dy)$ is uniformly bounded by $1/c_2(\alpha)$ in $s$ and in $\rho\in{\cal M}_1(E)$. Therefore, by~\eqref{eq:expo-cv},
\begin{align*}
\left|\E_x\left(\eta_s(X_t)\mid t<\tau_\d\right)- \alpha(\eta_s)\right|
\leq \frac{C}{c_2(\alpha)}e^{-\gamma t}.
\end{align*}
Since, $\alpha(\eta_s)=1$, we obtain 
\begin{align*}
\sup_{x\in E}\left|\eta_{t+s}(x)-\eta_t(x)\right|\leq \frac{C}{c_2(\alpha)^2}e^{-\gamma t}.
\end{align*}
That implies that $(\eta_t)_{t\geq 0}$ is a Cauchy family for the uniform norm and hence converges uniformly to a bounded limit $\eta$. By Lebesgue's theorem, $\alpha(\eta)=1$.

It only remains to prove that $\eta(x)>0$ for all $x\in E$. This is an immediate consequence of (A$2''$).

\medskip\noindent \textit{Step 2: Eigenfunction of the infinitesimal generator.}\\
We prove now that $\eta$ belongs to the domain of the infinitesimal generator $\cal L$ of the semi-group $(P_t)_{t\geq 0}$ and that
\begin{align}
\label{eq:Leta}
{\cal L}\eta=-\lambda_0\eta.
\end{align}
For any $h>0$, we have by the dominated convergence theorem and Step~1,
\begin{align*}
P_h\eta(x)&=\E_x\left(\eta(X_h)\right)
=\lim_{t\rightarrow\infty}\frac{\E_x\left(\P_{X_h}(t<\tau_\d)\right)}{\P_{\alpha}(t<\tau_\d)}.
\end{align*}
We have $\P_{\alpha}(t<\tau_\d)=e^{-\lambda_0 h}\P_{\alpha}(t+h<\tau_\d)$. Hence, by the Markov property,
\begin{align*}
P_h\eta(x) &= \lim_{t\rightarrow\infty}e^{-\lambda_0 h}\frac{\P_{x}(t+h<\tau_\d)}{\P_{\alpha}(t+h<\tau_\d)}\\
           &= e^{-\lambda_0 h} \eta(x).
\end{align*}
Since $\eta$ is uniformly bounded, it is immediate that
\begin{align*}
\frac{P_h\eta-\eta}{h}\xrightarrow[h\rightarrow 0]{\|\cdot\|_{\infty}}-\lambda_0 \eta.
\end{align*}
By definition of the infinitesimal generator, this implies that $\eta$ belongs to the domain of $\cal L$  and that~\eqref{eq:Leta} holds.

\subsection{Proof of Corollary~\ref{cor:2}}
\label{sec:proof-of-cor2}

Since $Lf=\lambda f$, we have $\E_x(f(X_t))=P_t f(x)=e^{\lambda t} f(x)$. When $f(\d)\neq 0$, taking $x=\d$, we see that $\lambda=0$ and, taking $x\not= \d$, the left hand side converges to $f(\d)$ and thus $f$ is constant.
 So let us assume that $f(\d)=0$. By property (v), 
\begin{align*}
\frac{P_t f(x)}{P_t\11_E(x)}-\alpha(f)\xrightarrow[t\rightarrow +\infty]{} 0
\end{align*}
uniformly in $x\in E$ and exponentially fast. Assume first that $\alpha(f)\neq 0$, then by Proposition~\ref{prop:2}, 
\begin{align*}
\frac{e^{(\lambda+\lambda_0)t} f(x)}{\eta(x)}\xrightarrow[t\rightarrow+\infty]{} \alpha(f),\ \forall x\in E.
\end{align*}
We deduce that $\lambda=-\lambda_0$ and $f(x)=\alpha(f)\eta(x)$ for all $x\in E\cup\{\d\}$. Assume finally that $\alpha(f)=0$, then, using~\eqref{eq:youpi2} to give a lower bound for $1/P_t\11_E(x)$, we deduce that
\begin{align*}
c_2(\alpha) e^{(\gamma+\lambda+\lambda_0)t}f(x)\leq \frac{e^{\gamma t}P_tf(x)}{P_t\11_E(x)},\ \forall x\in E,
\end{align*}
where the right hand side is bounded by property (v) of Theorem~\ref{thm:QSD_full}. Thus $\gamma+\lambda+\lambda_0\leq 0$.

\section{$Q$-process: proofs of the results of Section~\ref{sec:results-Q-process}}
\label{sec:proofs-Q-process}
We first prove Theorem~\ref{thm:Q-process} in Subsection~\ref{sec:pf-theorem-Q-process} and then Theorem~\ref{thm:generator-of-Q-process} in Subsection~\ref{sec:pf-thm-generator-of-Q-process}.

\subsection{Proof of Theorem~\ref{thm:Q-process}}
\label{sec:pf-theorem-Q-process}

\noindent \textit{Step 1: existence of the $Q$-process $\Q_x$ and expression of its transition kernel.}
We introduce $\Gamma_t=\11_{t<\tau_\d}$ and define the probability measure
\begin{align*}
Q^{\Gamma,x}_t=\frac{\Gamma_t}{\E_x\left(\Gamma_t\right)}\P_x,
\end{align*}
so that the $Q$-process exists if and only if $Q_t^{\Gamma,x}$ admits a proper limit when $t\rightarrow\infty$.
We have by the Markov property
\begin{align*}
\frac{\E_x\left(\Gamma_t\mid{\cal F}_s\right)}{\E_x\left(\Gamma_t\right)}=\frac{\11_{s<\tau_\d}\P_{X_s}\left(t-s<\tau_\d\right)}{\P_x\left(t<\tau_\d\right)}.
\end{align*}
By Proposition~\ref{prop:2}, this is uniformly bounded and converges almost surely to 
\begin{align*}
 M_s:=\11_{s<\tau_\d}e^{\lambda_0 s}\frac{\eta(X_s)}{\eta(x)}.
\end{align*}
By the dominated convergence Theorem, we obtain that
\begin{align*}
\E_x\left(M_s\right)=1.
\end{align*}
By the penalisation's theorem of Roynette, Vallois and Yor \cite[ Theorem~2.1]{Roynette_Vallois_Yor_2006}, these two conditions imply that $M$ is a martingale under $\P_x$ and that $Q_t^{\Gamma,x}(\Lambda_s)$ converges to $\E_x\left(M_s\11_{\Lambda_s}\right)$ for all $\Lambda_s\in{\cal F}_s$ when $t\rightarrow\infty$. This means that $\Q_x$ is well defined and 
\begin{align*}
\restriction{\frac{d\Q_x}{d\P_x}}{{\cal F}_s}=M_s.
\end{align*}
In particular, the transition kernel of the $Q$-process is given by
\begin{align*}
\tilde{p}(x;t,dy)=e^{\lambda_0 t}\frac{\eta(y)}{\eta(x)}p(x;t,dy).
\end{align*}

Let us now prove that $(\Q_x)_{x\in E}$ defines a Markov process, that is
\begin{align*}
\E_{\Q_x}\left(f(X_t)\mid {\cal F}_s\right)&=\E_{\Q_x}\left(f(X_t)\mid X_s\right),
\end{align*}
for all $s<t$ and all $f\in{\cal B}(E)$.
One easily checks from the definition of the conditional expectation that
\begin{align*}
M_s\E_{\Q_x}\left(f(X_t)\mid{\cal F}_s\right)&=\E_{x}\left(M_t f(X_t)\mid{\cal F}_s\right)
\end{align*}
By definition of $M_t$ and by the Markov property of $\P_x$, we deduce that
\begin{align*}
M_s \E_{\Q_x}\left(f(X_t)\mid{\cal F}_s\right)=\E_x(M_t f(X_t)\mid {X}_s).
\end{align*}
 Since 
$\E_{x}\left(M_t f(X_t)\mid X_s\right)=M_s \E_{\Q_x}\left(f(X_t)\mid X_s\right)$, this implies the Markov property for the $Q$-process.
The same proof gives the strong Markov property of $X$ under $(\Q_x)_{x\in E}$ provided that $X$ is strong Markov under $(\P_x)_{x\in E}$.

This concludes parts $i)$ and $ii)$ of Theorem~\ref{thm:Q-process}.

\medskip\noindent \textit{Step 2: exponential ergodicity of the $Q$-process.}\\
Let us first check that $\beta$ is invariant for $X$ under $\Q$. 
By~\eqref{eq:alphaPt} and \eqref{eq:semi-group-Q}, we have for all $t\geq 0$ and $\varphi\in{\cal B}(E)$
\begin{align*}
\alpha(\eta\tilde{P}_t\varphi)&=e^{\lambda_0 t}\alpha(P_t(\eta\varphi))=\alpha(\eta\varphi).
\end{align*}
Since $\eta$ is bounded, $\beta$ is well defined and is an invariant distribution for $X$ under $\Q$.

By~\eqref{eq:ergodicite1} and the semi-group property of $(R_{s,t}^T)_{s,t}$, we have for all $\mu_1,\mu_2\in{\cal M}_1(E)$ and all $t\leq T$
\begin{align*}
\left\|\mu_1 R^T_{0,t}-\mu_2 R^T_{0,t} \right\|_{TV}\leq (1-c_1c_2)^{\lfloor t/t_0\rfloor}\|\mu_1-\mu_2\|_{TV}.
\end{align*}
By definition of $R^T_{0,t}$ and by dominated convergence when $T\rightarrow\infty$, we obtain
\begin{align*}
\left\| \Q_{\mu_1}(X_t\in \cdot)-\Q_{\mu_2}(X_t\in \cdot)\right)\|_{TV}\leq (1-c_1c_2)^{\lfloor t/t_0\rfloor}\|\mu_1-\mu_2\|_{TV}.
\end{align*}
Taking $\mu_2=\beta$, this implies that $X$ is exponentially ergodic under $\Q$ with unique invariant distribution $\beta$.

\subsection{Proof of Theorem~\ref{thm:generator-of-Q-process}}
\label{sec:pf-thm-generator-of-Q-process}

\noindent\textit{Step 1: computation of $\tilde{L}^w$ and a first inclusion in \eqref{eq:domain-l-tilde}.}\\
Let us first show that
\begin{align*}
{\cal D}(\tilde{L}^w)\supset \left\{f\in{\cal B}(E),\;\eta f\in{\cal D}(L^w)\text{ and }\frac{L^w(\eta f)}{\eta}\text{is bounded}\right\}.
\end{align*}
Let $f$ belong to the set in the r.h.s. of the last equation. We have
\begin{align}
\label{eq:ptilde-utile}
\frac{\tilde{P}_h f-f}{h}&=\frac{e^{\lambda_0 h}}{\eta} \frac{P_h(\eta f)-\eta f}{h}+f\frac{e^{\lambda_0 h}-1}{h}.
\end{align}
So, obviously, $\frac{\tilde{P}_h f-f}{h}$ converges pointwise to $\tilde{L}^w f$. By \cite[I.1.15.C]{Dynkin1965},
\begin{align*}
P_t (\eta f)-\eta f=\int_0^t P_s L^w (\eta f) ds.
\end{align*}
Since $|L^w(\eta f)|\leq C\eta$ for some constant $C>0$ and $P_s \eta=e^{\lambda_0 s}\eta$, we obtain
\begin{align*}
\left|P_t (\eta f)-\eta f\right|\leq C \frac{e^{\lambda_0 t}-1}{\lambda_0}\eta.
\end{align*}
Therefore the r.h.s. of \eqref{eq:ptilde-utile} is uniformly bounded. Finally, $\tilde{L}^w f$ is the b.p.--limit of $\frac{\tilde{P}_h f-f}{h}$.

It only remains to check that b.p.-$\lim \tilde{P}_h\tilde{L}^w f=\tilde{L}^wf$. Since $\tilde{L}^wf$ is bounded, we only have to prove the pointwise convergence. We have
\begin{align*}
\tilde{P}_h\tilde{L}^w f&=\lambda_0 \tilde{P}_h f+\tilde{P}_h\frac{L^w(\eta f)}{\eta}= \lambda_0 \tilde{P}_h f+\frac{e^{\lambda_0 h}}{\eta}P_h L^w(\eta f)
\end{align*}
where the last equality comes from \eqref{eq:semi-group-Q}. Since $\eta f\in{\cal D}(L^w)$, the pointwise convergence is clear.

\noindent\textit{Step 2: the second inclusion in \eqref{eq:domain-l-tilde}.}\\
Let $f\in {\cal D}(\tilde{L}^w)$. We have b.p.--convergence in \eqref{eq:ptilde-utile} and, since $\eta$ and $f$ are bounded, $\frac{P_h (\eta f)-\eta f}{h}$ b.p.--converges to some limit $g\in{\cal B}(E\cup\{\d\})$ (recall that by convention $\eta f(\d)=0$) such that
\begin{align*}
g=\eta \tilde{L}^w f-\lambda_0 \eta f.
\end{align*}
Let us check that b.p.-$\lim P_h g=g$. Since $g$ is bounded, we only have to prove the pointwise convergence. We have
\begin{align*}
P_h g&=P_h(\eta \tilde{L}^w f)-\lambda_0 P_h(\eta f)=\eta e^{-\lambda_0 h} \tilde{P}_h \tilde{L}^w f-\lambda_0 P_h(\eta f),
\end{align*}
by \eqref{eq:semi-group-Q}. Since $f\in {\cal D}(\tilde{L}^w)$, the first term converges pointwise $\eta \tilde{L}^w f$. The second term converges to $-\lambda_0 \eta f$ since we have proved the b.p.--convergence of $\frac{P_h (\eta f)-\eta f}{h}$. We deduce that
\begin{align*}
\text{b.p.-}\lim_{h\rightarrow 0} P_h(g)= \eta \tilde{L}^w f-\lambda_0 \eta f =g.
\end{align*}
Thus $\eta f\in {\cal D}(L^w)$ and $L^w(\eta f)=\eta \tilde{L}^w f-\lambda_0 \eta f$, so that $ |L^w(\eta f)|\leq C\eta$ for some constant $C>0$.

\noindent\textit{Step 3: characterization of $(\tilde{P}_t)$ by $\tilde{L}^w$}\\
We assume that $E$ is a topological space and that $\cal E$ is the Borel $\sigma$-field. By \cite[Thm II.2.3]{Dynkin1965}, the result follows if
$(\tilde{P}_t)$ is stochastically continuous, \textit{i.e.}, for all open set $U\subset E$ and all $x\in U$,
\begin{align*}
\tilde{P}_h\11_U(x)\xrightarrow[h\rightarrow 0]{}1.
\end{align*}
By \eqref{eq:semi-group-Q}, we have
\begin{align*}
\tilde{P}_t\11_{U^c}(x)&=\frac{e^{\lambda_0t}}{\eta(x)}P_t(\eta\11_{U^c})(x)\leq \frac{e^{\lambda_0t}\|\eta\|_{\infty}}{\eta(x)}P_t(\11_{U^c})(x).
\end{align*}
The results follows from the assumption~\eqref{eq:stochasticallycontinuous}.

\bigskip
\noindent{\bf Acknowledgements:} We are grateful to the anonymous referee for constructive comments and suggestions.

\bibliographystyle{abbrv}
\bibliography{biblio-bio,biblio-denis,biblio-math,biblio-math-nicolas}

\end{document}